\documentclass[10pt,a4paper]{article}

\usepackage[T1]{fontenc}
\usepackage[utf8]{inputenc}

\usepackage{cite}

\usepackage{mathpazo} 
\usepackage[scaled]{helvet} 
\usepackage{courier} 
\normalfont

\usepackage{amsmath}
\usepackage{amssymb}
\usepackage{mathtools}
\usepackage{amsthm}
\usepackage{tikz-cd}
\usepackage[hidelinks]{hyperref}	
\usepackage{graphicx}
\usepackage{subfig}
\usepackage{siunitx}
\usepackage{authblk}	
\usepackage{pinlabel}

\newtheorem{theorem}{Theorem}[section]
\newtheorem{lemma}[theorem]{Lemma}
\newtheorem{proposition}[theorem]{Proposition}

\theoremstyle{definition}
\newtheorem{definition}[theorem]{Definition}
	\newtheorem{assumption}[theorem]{Assumption}

\newtheorem{remark}[theorem]{Remark}
\usepackage{xcolor}

\providecommand{\MSC}[1]{\par\smallskip\textbf{\textbf{MSC2020: }} #1}
\providecommand{\keywords}[1]{\par\smallskip\textbf{\textbf{Keywords: }} #1}

\usepackage{booktabs}

\usepackage{tikz}  
\usetikzlibrary{matrix} 
\pgfdeclarelayer{background} 
\pgfsetlayers{background,main} 
\usepackage{verbatim}
\usepackage[normalem]{ulem}
\numberwithin{equation}{section}
\usepackage{todonotes}
\usepackage[normalem]{ulem}
\usepackage{bbm}
\usepackage{cancel}

\usepackage[shortlabels]{enumitem}

\newcommand{\cA}{\mathcal{A}}

\newcommand{\cC}{\mathcal{C}}
\newcommand{\cD}{\mathcal{D}}

\newcommand{\cG}{\mathcal{G}}
\newcommand{\cH}{\mathcal{H}}
\newcommand{\cI}{\mathcal{I}}

\newcommand{\cK}{\mathcal{K}}
\newcommand{\cL}{\mathcal{L}}

\newcommand{\cP}{\mathcal{P}}

\newcommand{\cU}{\mathcal{U}}
\newcommand{\cV}{\mathcal{V}}

\newcommand{\cX}{\mathcal{X}}
\newcommand{\cY}{\mathcal{Y}}

\newcommand{\bR}{\mathbb{R}}

\newcommand{\bfH}{\mathbf{H}}

\newcommand{\vn}[1]{\left| \! \left| #1\right| \! \right|} 

\newcommand{\ip}[2]{\langle #1,#2\rangle}

\newcommand{\bONE}{\mathbbm{1}}
\newcommand{\dd}{ \mathrm{d}}

\title{Comparison Principle for Hamilton-Jacobi-Bellman Equations via a Bootstrapping Procedure}
\author{Richard C. Kraaij\thanks{Delft Institute of Applied Mathematics, Delft University of Technology, Van Mourik Broekmanweg 6, 2628 XE Delft, The Netherlands. \emph{E-mail address}: r.c.kraaij@tudelft.nl} \quad \quad Mikola C. Schlottke\thanks{Department of Mathematics and Computer Science, Eindhoven University of Technology. \emph{E-mail address}: mikola.schlottke@outlook.com}}

\begin{document}
	
	\maketitle
	{\abstract{
			
			We study the well-posedness of Hamilton-Jacobi-Bellman equations on subsets of $\mathbb{R}^d$ in a context without boundary conditions. The Hamiltonian is given as the supremum over two parts: an internal Hamiltonian depending on an external control variable and a cost functional penalizing the control.
			
			The key feature in this paper is that the control function can be unbounded and discontinuous. This way we can treat functionals that appear e.g. in the Donsker-Varadhan theory of large deviations for occupation-time measures.
			To allow for this flexibility, we assume that the internal Hamiltonian and cost functional have controlled growth, and that they satisfy an equi-continuity estimate uniformly over compact sets in the space of controls.
			
			\smallskip
			
			In addition to establishing the comparison principle for the Hamilton-Jacobi-Bellman equation, we also prove existence, the viscosity solution being the value function with exponentially discounted running costs.
			
			\smallskip
			
			As an application, we verify the conditions on the internal Hamiltonian and cost functional in two examples.
			
			
			\MSC{49L25, 35F21} 
			
	}}
	
	\section{Introduction and aim of this note} \label{section:introduction} \label{section:introductory_example}
	The main purpose of this note is to establish well-posedness for first-order nonlinear partial differential equations of Hamilton-Jacobi-Bellman type on subsets $E$ of $\mathbb{R}^d$,
	\begin{equation}\label{eq:intro:HJ-general}
		u(x)-\lambda \,\mathcal{H}\left(x,\nabla u(x)\right) = h(x),\quad x \in E\subseteq \mathbb{R}^d,\tag{\text{HJB}}
	\end{equation}
	in the context without boundary conditions and where the Hamiltonian flow generated by $\cH$ remains inside $E$. In \eqref{eq:intro:HJ-general}, $\lambda > 0$ is a scalar and  $h$ is a continuous and bounded function. The Hamiltonian $\mathcal{H}:E\times \mathbb{R}^d\to \mathbb{R}$ is given by
	\begin{equation}\label{eq:intro:variational_hamiltonian}
		\mathcal{H}(x,p) = \sup_{\theta \in \Theta}\left[\Lambda(x,p,\theta) - \mathcal{I}(x,\theta)\right],
	\end{equation}
	where $\theta \in \Theta$ plays the role of a control variable. For fixed $\theta$, the function $\Lambda$ can be interpreted as an Hamiltonian itself. We call it the \emph{internal} Hamiltonian. The function $\cI$ can be interpreted as the cost of applying the control $\theta$.
	
	\smallskip
	
	The main result of this paper is the \emph{comparison principle} for \eqref{eq:intro:HJ-general}  in order to establish uniqueness of viscosity solutions.  The standard assumption in the literature that allows one to obtain the comparison principle in the context of optimal control problems (e.g.~\cite{BaCD97} for the first order case and \cite{DLLe11} for the second order case) is that either there is a modulus of continuity $\omega$ such that
	\begin{equation}\label{eq:intro:standard-reg-estimate-on-H}
		|\mathcal{H}(x,p)-\mathcal{H}(y,p)| \leq \omega\left(|x-y|(1+|p|)\right),
	\end{equation} 
	or that $\cH$ is uniformly coercive:
	\begin{equation} \label{eqn:uniformly_coercive}
		\lim_{|p| \rightarrow \infty} \inf_x \cH(x,p) = \infty.
	\end{equation}	
	More generally, the two estimates \eqref{eq:intro:standard-reg-estimate-on-H} and \eqref{eqn:uniformly_coercive} can be combined in a single estimate, called pseudo-coercivity, see \cite[(H4), Page 34]{Ba94}, that uses the fact that the sub- and supersolution properties roughly imply that the estimate \eqref{eq:intro:standard-reg-estimate-on-H} only needs to hold for appropriately chosen $x,y$ and $p$ such that $\cH$ is finite uniformly over these chosen $x,y,p$.

	\smallskip
	
	In the Hamilton-Jacobi-Bellman context, the comparison principle is typically obtained by translating \eqref{eq:intro:standard-reg-estimate-on-H} into  conditions for $\Lambda$ and $\cI$ of~\eqref{eq:intro:variational_hamiltonian}, which include (e.g. \cite[Chapter~III]{BaCD97})
	\begin{enumerate}[(I)]
		\item \label{item:intro_conditions_old_continuity_estimate} $|\Lambda(x,p,\theta)-\Lambda(y,p,\theta)|\leq \omega_\Lambda(|x-y|(1+|p|))$, uniformly in $\theta$, and
		\item \label{item:intro_conditions_old_equi_continuity} $\cI$ is bounded, continuous and $|\cI(x,\theta) - \cI(y,\theta)| \leq \omega_\cI(|x-y|)$ for all~$\theta$.
	\end{enumerate}
	The pseudo-coercivity property is harder to translate as in this way the control on $\cH$ does not necessarily imply the same control on $\Lambda$, in particular in the case when $\cI$ is unbounded. We return on this issue below.
	
	\smallskip
	
	The estimates \ref{item:intro_conditions_old_continuity_estimate} and \ref{item:intro_conditions_old_equi_continuity} are not satisfied for Hamiltonians arising from natural examples in the theory of large deviations \cite{Ho08,DZ98} for Markov processes with two scales (see e.g. \cite{BuDuGa18,FeFoKu12,KuPo17,PeSc19} for PDE's arising from large deviations with two scales, see \cite{BaCeGh15,Gh18,DuIiSo90,FlSo86,EvSo89} for other works connection PDE's with large deviations). Indeed, in \cite{BuDuGa18} the authors mention that well-posedness of the Hamilton-Jacobi-Bellman equation for examples arising from large deviation theory is an open problem. Recent generalizations of the coercivity condition, see e.g. \cite{CuDL07}, also do not cover these examples. 
	
	\smallskip
	
	In the large deviation context, however, we typically know that we have the comparison principle for the Hamilton-Jacobi equation in terms of $\Lambda$. In addition, even though $\cI$ might be discontinuous, we do have other types of regularity for the  functional $\cI$, see e.g. \cite{Pi07}. Thus, we aim to prove a comparison principle for \eqref{eq:intro:HJ-general} on the basis of the \textit{assumption} that we have the following natural relaxations of (or the pseudo-coercive version of) \ref{item:intro_conditions_old_continuity_estimate} and \ref{item:intro_conditions_old_equi_continuity}.
	
	\begin{enumerate}[(i)]
		\item \label{item:intro_conditions_continuity_estimate} 
		For $\theta \in \Theta$, define the Hamiltonian  $\cH_\theta(x,p):= \Lambda(x,p,\theta)$.
		We have an estimate on $\cH_\theta$ that is uniform over $\theta$ in compact sets $K \subseteq \Theta$.
		This estimate, for one fixed $\theta$, is in spirit similar to the pseudo-coercivity estimate of \cite{Ba94} and is morally equivalent to the comparison principle for $\cH_\theta$.
		The uniformity is made rigorous as the \textit{continuity estimate} in Assumption \ref{assumption:results:regularity_of_V} \ref{item:assumption:slow_regularity:continuity_estimate} below. 
		\item \label{item:intro_conditions_equi_continuity} The cost functional $\mathcal{I}(x,\theta)$ satisfies an \textit{equi-continuity estimate} of the type $|\cI(x,\theta) - \cI(y,\theta)| \leq \omega_{\cI,C}(|x-y|)$ on sublevel sets $\{\mathcal{I} \leq C\}$ which we assume to be compact.
		This estimate is made rigorous in Assumption \ref{assumption:results:regularity_I} \ref{item:assumption:I:equi-cont} below.
	\end{enumerate}
	
	To work with these relaxations, we introduce a procedure that allows us to restrict our analysis to compact sets in the space of controls. In the proof of the comparison principle, the sub- and supersolution properties give boundedness of $\cH$ when evaluated in optimizing points. We then translate this boundedness to boundedness of $\cI$, which implies that the controls lie in a compact set.
	
	\smallskip
	
	The transfer of control builds upon \ref{item:intro_conditions_continuity_estimate} for $\Lambda(x,p,\theta_{x}^0)$ when we use a control $\theta_{x}^0$ that satisfies $\cI(x,\theta_{x}^0) = 0$. This we call the \textit{bootstrap procedure}: we use the comparison principle for the Hamilton-Jacobi equation in terms of $\Lambda(x,p,\theta_{x}^0)$ to shift the control on $\cH$ to control on $\Lambda$ and $\cI$ for general $\theta$. That way the comparison principle for the internal Hamiltonian~$\Lambda$ bootstraps to the comparison principle for the full Hamiltonian~$\mathcal{H}$.
	
	\smallskip
	
	Clearly, this bootstrap argument does not come for free. We pose four additional assumptions: 
	\begin{enumerate}[(i),resume]
		\item \label{item:intro_conditions_equal_growth_Lambda} The function $\Lambda$ grows roughly equally fast in $p$: For all compact sets $\widehat{K} \subseteq E$, there are constants $M,C_1,C_2$ such that
		\begin{equation*}
			\Lambda(x,p,\theta_1) \leq \max\left\{M,C_1 \Lambda(x,p,\theta_2) + C_2\right\}, 
		\end{equation*}
		for all~$x\in\widehat{K}$,~$p \in \bR^d, \, \theta_1,\theta_2 \in \Theta$.
		\item \label{item:intro_conditions_equal_growth_I} The function $\cI$ grows roughly equally fast in $x$: For all $x \in E$ and $M \geq 0$ there exists an open neighbourhood $U$ of $x$ and constants $M',C_1',C_2'$ such that 
		\begin{equation*}
			\cI(y_1,\theta) \leq \max\{M',C_1' \cI(y_2,\theta) + C_2'\}
		\end{equation*}
		for all $y_1,y_2 \in U$ and for all $\theta$ such that $\cI(x,\theta) \leq M$.
		\item \label{item:intro_zero_control} $\cI \geq 0$ and for each $x \in E$, there exists $\theta_{x}^0$ such that $\cI(x,\theta_x^0) = 0$.
		\item \label{item:intro_conditions_coercivity} The functional $\cI$ is equi-coercive in $x$: for any compact set $\hat{K} \subseteq E$ the set $\bigcup_{x \in \hat{K}} \{\theta \, | \, \cI(x,\theta) \leq C\}$ is compact.
	\end{enumerate}	
	These four assumptions are stated below as Assumptions \ref{assumption:results:regularity_of_V} \ref{item:assumption:slow_regularity:controlled_growth}, \ref{assumption:results:regularity_I} \ref{item:assumption:I:finiteness}, \ref{assumption:results:regularity_I} \ref{item:assumption:I:zero-measure}, and \ref{assumption:results:regularity_I} \ref{item:assumption:I:compact-sublevelsets}.
	To explain in more detail our argument, we give a sketch of the bootstrap procedure, which can be skipped on first reading.
	In this sketch, we refrain from performing localization arguments that are needed for non-compact $E$.
	
	\smallskip
	
	\textbf{Sketch of the bootstrap argument}
	
	\smallskip
	
	
	Let $u$ and~$v$ be a sub- and supersolution to $f - \lambda Hf = h$ respectively. We estimate $\sup_x u(x) - v(x)$ by the classical doubling of variables by means of penalizing the distance between $x$ and $y$ by some penalization $\alpha\Psi(x-y)$ and aim to send $\alpha \rightarrow \infty$. Let $x_\alpha,y_\alpha$ denote the optimizers, and denote by $p_\alpha$ the corresponding momentum $p_\alpha = \alpha \partial_x \Psi(x_\alpha -y_\alpha)$. Let $\theta_\alpha$ be the control such that $\cH(x_\alpha,p_\alpha) = \Lambda(x_\alpha,p_\alpha,\theta_\alpha)-\mathcal{I}(x_\alpha,\theta_\alpha)$ and let $\theta_{\alpha}^0$ be a control such that $\cI(y_\alpha,\theta_{\alpha}^0) = 0$, which exists due to \ref{item:intro_zero_control}.
	
	\smallskip
	
	The supersolution property for $v$ yields the following estimate that is uniform in $\alpha > 0$
	\begin{equation} \label{eqn:intro_explain_bootstrap_control_Lambda_initial}
		\infty > \frac{\vn{v - h}}{\lambda} \geq \cH(y_\alpha,p_\alpha) \geq \Lambda(y_\alpha,p_\alpha,\theta_{\alpha}^0) - \cI(y_\alpha, \theta_{\alpha}^0) = \Lambda(y_\alpha,p_\alpha,\theta_{\alpha}^0).
	\end{equation}
	Using \ref{item:intro_conditions_equal_growth_Lambda}, we obtain a uniform estimate in $\alpha$: 
	\begin{equation} \label{eqn:intro_explain_bootstrap_control_Lambda}
		\sup_\alpha\Lambda(y_\alpha,p_\alpha,\theta_\alpha) < \infty.
	\end{equation}
	which will allow us to use \ref{item:intro_conditions_continuity_estimate} if we can show that the controls $\theta_\alpha$ take their value in a compact set $K \subseteq \Theta$. For this, it suffices by \ref{item:intro_conditions_coercivity} to establish 
	\begin{equation}  \label{eqn:intro_explain_bootstrap_control_on_cI}
		\sup_\alpha \cI(x_\alpha,\theta_\alpha) < \infty.
	\end{equation}
	This, in fact, implies by \ref{item:intro_conditions_equal_growth_I} that
	\begin{equation*} 
		\sup \cI(y_\alpha,\theta_\alpha) \vee\cI(x_\alpha,\theta_\alpha) < \infty
	\end{equation*}
	so that we can also apply \ref{item:intro_conditions_equi_continuity}. This, in combination with the application of \ref{item:intro_conditions_continuity_estimate} establishes the comparison principle for $f - \lambda Hf = h$.
	
	\smallskip
	
	We are thus left to prove \eqref{eqn:intro_explain_bootstrap_control_on_cI}, which is where our bootstrap comes into play. The subsolution property for $u$ yields the following estimate that is uniform in $\alpha > 0$ 
	\begin{equation*}
		- \infty < \frac{\vn{u - h}}{\lambda} \leq \cH(x_\alpha,p_\alpha) = \Lambda(x_\alpha,p_\alpha,\theta_\alpha) - \cI(x_\alpha,\theta_\alpha).
	\end{equation*}
	Thus, \eqref{eqn:intro_explain_bootstrap_control_on_cI} follows if we can establish
	\begin{equation*}
		\sup_\alpha \Lambda(x_\alpha,p_\alpha,\theta_\alpha) < \infty,
	\end{equation*}
	which in turn (by \ref{item:intro_conditions_equal_growth_Lambda}) follows from
	\begin{equation*}
		\sup_\alpha \Lambda(x_\alpha,p_\alpha,\theta_{\alpha}^0) < \infty.
	\end{equation*}
	To establish this final estimate, note that
	\begin{equation*}
		\Lambda(x_\alpha,p_\alpha,\theta_{\alpha}^0) = \Lambda(y_\alpha,p_\alpha,\theta_{\alpha}^0) + \left[\Lambda(x_\alpha,p_\alpha,\theta_{\alpha}^0) - \Lambda(y_\alpha,p_\alpha,\theta_{\alpha}^0) \right]
	\end{equation*}
	and that we have control on the first term by means of \eqref{eqn:intro_explain_bootstrap_control_Lambda_initial} and on the second term by the pseudo-coercivity estimate of \ref{item:intro_conditions_continuity_estimate} on $\Lambda$ for the controls $\theta_{\alpha}^0$ which lie in a compact set due to \ref{item:intro_conditions_coercivity}.
	\qed
	
	\smallskip
	
	Thus, to summarize, we use the growth conditions posed on~$\Lambda$ and~$\cI$ and the pseudo-coercivity estimate for~$\Lambda$ to transfer the control on the full Hamiltonian~$\mathcal{H}$ to the functions~$\Lambda$ and the cost function~$\mathcal{I}$. Then the control on~$\Lambda$ and~$\mathcal{I}$ allows us to apply the estimates~\ref{item:intro_conditions_continuity_estimate} and~\ref{item:intro_conditions_equi_continuity} to obtain the comparison principle.

	\smallskip
	
	Next to our main result, we also state for completeness an existence result in Theorem \ref{theorem:existence_of_viscosity_solution}. The viscosity solution will be given in terms of a discounted control problem as is typical in the literature, see e.g. \cite[Chapter 3]{BaCD97}. Minor difficulties arise from working with $\cH$ that arise from irregular $\cI$.
	
	\smallskip 
	
	Finally, we show that the conditions \ref{item:intro_conditions_continuity_estimate} to \ref{item:intro_conditions_coercivity} are satisfied in two examples that arise from large deviation theory for two-scale processes. In our companion paper \cite{KrSchl20}, we will use existence and uniqueness for \eqref{eq:intro:HJ-general} for these examples to obtain large deviation principles.

	\smallskip

	\textbf{Illustration in the context of an example}
	
	\smallskip
	
	As an illustrating example, we consider a Hamilton-Jacobi-Bellman equation that arises from the large deviations of the empirical measure-flux pair of weakly coupled Markov jump processes that are coupled to fast Brownian motion on the torus. We skip the probabilistic background of this problem (See \cite{KrSchl20}), and come to the set-up relevant for this paper.
	
	\smallskip 
	
	Let $G := \{1,\dots,q\}$ be some finite set, and let $\Gamma = \{(a,b) \in G^2 \, | \, a \neq b\}$ be the set of directed bonds. Let $E := \cP(G) \times [0,\infty)^\Gamma$, where $\cP(G)$ is the set of probability measures on $G$. Let $F = \cP(S^1)$ be the set of probability measures on the one-dimensional torus. We introduce $\Lambda$ and $\cI$.
	
	\begin{itemize}
		\item 	Let $r : G \times G \times \cP(E) \times \cP(S^1) \rightarrow [0,\infty)$ be some function that codes the $\cP(E) \times \cP(S^1)$ dependent jump rate of the Markov jump process over each bond $(a,b) \in \Gamma$. The internal Hamiltonian  $\Lambda$ is given by
		\begin{equation*}
			\Lambda(\mu,p,\theta) = \sum_{(a,b) \in \Gamma} \mu_a r(a,b,\mu,\theta) \left[e^{p_b - p_a + p_{a,b}} - 1 \right].
		\end{equation*}
		\item Let $\sigma^2 : S^1 \times \cP(G) \rightarrow (0,\infty)$ be a bounded and strictly positive function. The cost function~$\mathcal{I}:E \times \Theta\to[0,\infty]$ is given by
		\begin{equation*}
			\mathcal{I}(\mu,w,\theta) = \cI(\mu,\theta) = \sup_{\substack{u\in C^\infty(S^1)\\ u > 0}} \int_{S^1} \sigma^2(y,\mu) \left(-\frac{u''(y)}{u(y)}\right)\,\theta(\dd y).
		\end{equation*}
	\end{itemize}
	Aiming for the comparison principle, we note that classical methods do not apply. 
	The functionals~$\Lambda$ are not coercive and do not satisfy~\ref{item:intro_conditions_old_continuity_estimate}. 
	We show in Appendix~\ref{appendix:pseudo_coercive} that they are also not pseudo-coercive as defined in~\cite{Ba94}. The functional $\cI$ is neither continuous nor bounded. Once can check e.g. that if $\theta$ is a finite combination of Dirac measures, then $\cI(\mu,\theta) = \infty$.
	
	\smallskip
	
	We show in Section \ref{section:verification-for-examples-of-Hamiltonians}, however, that \ref{item:intro_conditions_continuity_estimate} to \ref{item:intro_conditions_coercivity} hold, implying the comparison principle for the Hamilton-Jacobi-Bellman equations. The verification of these properties is based in part on results from \cite{Kr17,Pi07}.

	\smallskip

	\textbf{Summary and overview of the paper}
	
	\smallskip
	
	To summarize, our novel bootstrap procedure allows to treat Hamilton-Jacobi-Bellman equations where:
	\begin{itemize}
		\item We assume that the cost function~$\cI$ satisfies some regularity conditions on its sub-levelsets, but allow~$\cI$ to be possibly unbounded and discontinuous.  
		\item We assume that $\Lambda$ satisfies the continuity estimate uniformly for controls in compact sets, which in spirit extends the pseudo-coercivity estimate of \cite{Ba94}.
		This implies that $\Lambda$ can be possibly non-coercive, non-pseudo-coercive and non-Lipschitz as exhibited in our example above.
	\end{itemize}
	
	In particular, allowing discontinuity in $\cI$ allows us to treat the comparison principle for examples like the one we considered above, which so far has been out of reach.
	We believe that the \emph{bootstrap procedure} we introduce in this note has the potential to also apply to second order equations or equations in infinite dimensions. Of interest would be, for example, an extension of the results of \cite{DLLe11} who work with continuous $\cI$. For clarity of the exposition, and the already numerous applications for this setting, we stick to the finite-dimensional first-order case. We think that the key arguments that are used in the proof in Section \ref{section:comparison_principle} do not depend in a crucial way on this assumption.
	
	\smallskip
	
	The paper is organized as follows. The main results are formulated in Section \ref{section:results}. In Section \ref{section:comparison_principle} we establish the comparison principle. In Section \ref{section:construction-of-viscosity-solutions} we establish that a resolvent operator $R(\lambda)$ in terms of an exponentially discounted control problem gives rise to viscosity solutions of the  Hamilton-Jacobi-Bellman equation~\eqref{eq:intro:HJ-general}. Finally, in Section~\ref{section:verification-for-examples-of-Hamiltonians} we treat two examples including the one mentioned in the introduction.

	\section{Main Results}
	\label{section:results}
	In this section, we start with preliminaries in Section~\ref{section:preliminaries}, which includes the definition of viscosity solutions and that of the comparison principle. 
	
	\smallskip 
	
	We proceed in Section~\ref{section:results:HJ-of-Perron-Frobenius-type} with the main results: a comparison principle for the Hamilton-Jacobi-Bellman equation~\eqref{eq:intro:HJ-general} based on variational Hamiltonians of the form~\eqref{eq:intro:variational_hamiltonian}, and the existence of viscosity solutions. 
	In Section \ref{section:assumptions} we collect all assumptions that are needed for the main results.
	\subsection{Preliminaries} \label{section:preliminaries}
	For a Polish space $\cX$ we denote by $C(\cX)$ and $C_b(\cX)$ the spaces of continuous and bounded continuous functions respectively. If $\cX \subseteq \bR^d$ then we denote by $C_c^\infty(\cX)$ the space of smooth functions that vanish outside a compact set. We denote by $C_{cc}^\infty(\cX)$ the set of smooth functions that are constant outside of a compact set in $\cX$, and by $\cP(\cX)$ the space of probability measures on $\cX$. We equip $\cP(\cX)$ with the weak topology induced by convergence of integrals against bounded continuous functions.
	
	\smallskip
	
	Throughout the paper, $E$ will be the set on which we base our Hamilton-Jacobi equations. We assume that $E$ is a subset of $\bR^d$ that is a Polish space which is contained in the $\bR^d$ closure of its $\bR^d$ interior. This ensures that gradients of functions are determined by their values on $E$. Note that we do not necessarily assume that $E$ is open. We assume that the space of controls $\Theta$ is Polish.
	
	\smallskip
	
	We next introduce viscosity solutions for the Hamilton-Jacobi equation with Hamiltonians like $\mathcal{H}(x,p)$ of our introduction.
	\begin{definition}[Viscosity solutions and comparison principle] \label{definition:viscosity_solutions}
		Let $A : \cD(A) \subseteq C_b(E) \to C_b(E)$ be an operator with domain $\mathcal{D}(A)$, $\lambda > 0$ and $h \in C_b(E)$. Consider the Hamilton-Jacobi equation
		\begin{equation}
			f - \lambda A f = h. \label{eqn:differential_equation} 
		\end{equation}
		We say that $u$ is a \textit{(viscosity) subsolution} of equation \eqref{eqn:differential_equation} if $u$ is bounded from above, upper semi-continuous and if, for every $f \in \cD(A)$ there exists a sequence $x_n \in E$ such that
		\begin{gather*}
			\lim_{n \uparrow \infty} u(x_n) - f(x_n)  = \sup_x u(x) - f(x), \\
			\lim_{n \uparrow \infty} u(x_n) - \lambda A f(x_n) - h(x_n) \leq 0.
		\end{gather*}
		We say that $v$ is a \textit{(viscosity) supersolution} of equation \eqref{eqn:differential_equation} if $v$ is bounded from below, lower semi-continuous and if, for every $f \in \cD(A)$there exists a sequence $x_n \in E$ such that
		\begin{gather*}
			\lim_{n \uparrow \infty} v(x_n) - f(x_n)  = \inf_x v(x) - f(x), \\
			\lim_{n \uparrow \infty} v(x_n) - \lambda Af(x_n) - h(x_n) \geq 0.
		\end{gather*}
		We say that $u$ is a \textit{(viscosity) solution} of equation \eqref{eqn:differential_equation} if it is both a subsolution and a supersolution to \eqref{eqn:differential_equation}.
		We say that \eqref{eqn:differential_equation} satisfies the \textit{comparison principle} if for every subsolution $u$ and supersolution $v$ to \eqref{eqn:differential_equation}, we have $u \leq v$.
	\end{definition}
	\begin{remark}[Uniqueness]
		If $u$ and $v$ are two viscosity solutions of~\ref{eq:results:HJ-eq}, then we have $u\leq v$ and $v\leq u$ by the comparison principle, giving uniqueness.
	\end{remark}
	\begin{remark} \label{remark:existence of optimizers}
		Consider the definition of subsolutions. Suppose that the testfunction $f \in \cD(A)$ has compact sublevel sets, then instead of working with a sequence $x_n$, there exists $x_0  \in E$ such that
		\begin{gather*}
			u(x_0) - f(x_0)  = \sup_x u(x) - f(x), \\
			u(x_0) - \lambda A f(x_0) - h(x_0) \leq 0.
		\end{gather*}
		A similar simplification holds in the case of supersolutions.
	\end{remark}
	\begin{remark}
		For an explanatory text on the notion of viscosity solutions and fields of applications, we refer to~\cite{CIL92}.
	\end{remark}
	
	\begin{remark}
		At present, we refrain from working with unbounded viscosity solutions as we use the upper bound on subsolutions and the lower bound on supersolutions in the proof of Theorem \ref{theorem:comparison_principle_variational}. We can, however, imagine that the methods presented in this paper can be generalized if $u$ and $v$ grow slower than the containment function $\Upsilon$ that will be defined below in Definition \ref{def:results:compact-containment}.
	\end{remark}
	
	\subsection{Main results: comparison and existence}
	\label{section:results:HJ-of-Perron-Frobenius-type}
	In this section, we state our main results: the comparison principle in Theorem \ref{theorem:comparison_principle_variational}, and existence of solutions in Theorem \ref{theorem:existence_of_viscosity_solution}.
	\smallskip
	
	Consider the variational Hamiltonian $\cH : E \times \bR^d \rightarrow  \bR$ given by
	\begin{equation}\label{eq:results:variational_hamiltonian}
		\mathcal{H}(x,p) = \sup_{\theta \in \Theta}\left[\Lambda(x,p,\theta) - \mathcal{I}(x,\theta)\right].
	\end{equation}
	The precise assumptions on the maps $\Lambda$ and $\mathcal{I}$ are formulated in Section~\ref{section:assumptions}. 
	\begin{theorem}[Comparison principle]
		\label{theorem:comparison_principle_variational}
		Consider the map $\cH : E \times \bR^d \rightarrow \bR$ as in \eqref{eq:results:variational_hamiltonian}. Suppose that Assumptions~\ref{assumption:results:regularity_of_V} and~\ref{assumption:results:regularity_I} are satisfied for $\Lambda$ and $I$.
		Define the operator $\bfH f(x) := \cH(x,\nabla f(x))$ with domain $\cD(\bfH) = C_{cc}^\infty(E)$. Then:
		\begin{enumerate}[(a)]
			\item For any $f \in \cD(\bfH)$ the map $x\mapsto \bfH f(x)$ is continuous.
			\item For any $h \in C_b(E)$ and $\lambda > 0$, the comparison principle holds for
			\begin{equation}\label{eq:results:HJ-eq}
				f - \lambda \, \bfH f = h.
			\end{equation}
		\end{enumerate}
	\end{theorem}
	\begin{remark}[Domain]
		The comparison principle holds with any domain that satisfies $C_{cc}^\infty(E)\subseteq \mathcal{D}(\mathbf{H})\subseteq C^1_b(E)$. We state it with $C^\infty_{cc}(E)$ to connect it with the existence result of Theorem~\ref{theorem:existence_of_viscosity_solution}, where we need to work with test functions whose gradients have compact support.
	\end{remark}
	Consider the Legendre dual $\cL : E \times \bR^d \rightarrow [0,\infty]$ of the Hamiltonian,
	\begin{equation*}
		\cL(x,v) := \sup_{p\in\mathbb{R}^d} \left[\ip{p}{v} - \cH(x,p)\right],
	\end{equation*}
	and denote the collection of absolutely continuous paths in $E$ by $\cA\cC$.
	\begin{theorem}[Existence of viscosity solution] \label{theorem:existence_of_viscosity_solution}
		Consider $\cH : E \times \bR^d \rightarrow \bR$ as in \eqref{eq:results:variational_hamiltonian}. Suppose that Assumptions~\ref{assumption:results:regularity_of_V} and~\ref{assumption:results:regularity_I} are satisfied for $\Lambda$ and~$\mathcal{I}$, and that $\mathcal{H}$ satisfies Assumption~\ref{assumption:Hamiltonian_vector_field}. For each $\lambda > 0$, let $R(\lambda)$ be the operator
		\begin{equation*}
			R(\lambda) h(x) = \sup_{\substack{\gamma \in \mathcal{A}\mathcal{C}\\ \gamma(0) = x}} \int_0^\infty \lambda^{-1} e^{-\lambda^{-1}t} \left[h(\gamma(t)) - \int_0^t \mathcal{L}(\gamma(s),\dot{\gamma}(s))\right] \, \dd t.
		\end{equation*}
		Then $R(\lambda)h$ is the unique viscosity solution to $f - \lambda \bfH f = h$.
	\end{theorem}
	\begin{remark}
		The form of the solution is typical, see for example Section III.2 in \cite{BaCD97}. It is the value function obtained by an optimization problem with exponentially discounted cost. The difficulty of the proof of Theorem~\ref{theorem:existence_of_viscosity_solution} lies in treating the irregular form of $\cH$. 
	\end{remark}
	\subsection{Assumptions} \label{section:assumptions}
	In this section, we formulate and comment on the assumptions imposed on the Hamiltonians defined in the previous sections. 
	The key assumptions were already mentioned in the sketch of the bootstrap method in the introduction. To these, we add minor additional assumptions on the regularity of $\Lambda$ and $\cI$ in Assumptions \ref{assumption:results:regularity_of_V} and \ref{assumption:results:regularity_I}. 
	Finally, Assumption \ref{assumption:Hamiltonian_vector_field} will imply that even if $E$ has a boundary, no boundary conditions are necessary for the construction of the viscosity solution.
	\smallskip
	
	We start with the \emph{continuity estimate}  for $\Lambda$, which was briefly discussed in \ref{item:intro_conditions_continuity_estimate} in the introduction. 
	To that end, we first introduce a function that is used in the typical argument that doubles the number of variables.
	
	\begin{definition}[Penalization function]\label{def:results:good_penalization_function}
		We say that $\Psi : E^2 \rightarrow [0,\infty)$ is a \textit{ penalization function} if $\Psi \in C^1(E^2)$ and if $x = y$ if and only if $\Psi(x,y) = 0$.
	\end{definition}
	
	We will apply the definition below for $\cG = \Lambda$.
	
	\begin{definition}[Continuity estimate] \label{def:results:continuity_estimate}
		Let  $\Psi$ be a penalization function and let $\cG: E \times \mathbb{R}^d\times\Theta \rightarrow \bR$, $(x,p,\theta)\mapsto \cG(x,p,\theta)$ be a function. Suppose that for each $\varepsilon > 0$, 
		there is a sequence of positive real numbers $\alpha \rightarrow \infty$. For sake of readability, we suppress the dependence on $\varepsilon$ in our notation.
		
		Suppose that for each $\varepsilon$ and $\alpha$ we have variables $(x_{\varepsilon,\alpha},y_{\varepsilon,\alpha})$ in $E^2$ and variables $\theta_{\varepsilon,\alpha}$ in $\Theta$. We say that this collection is \textit{fundamental} for $\cG$ with respect to $\Psi$ if:
		\begin{enumerate}[label = (C\arabic*)]
			\item \label{item:def:continuity_estimate:1} For each $\varepsilon$, there are compact sets $K_\varepsilon \subseteq E$ and $\widehat{K}_\varepsilon\subseteq\Theta$ such that for all $\alpha$ we have $x_{\varepsilon,\alpha},y_{\varepsilon,\alpha} \in K_\varepsilon$ and $\theta_{\varepsilon,\alpha}\in\widehat{K}_\varepsilon$.
			\item \label{item:def:continuity_estimate:2} 
			For each $\varepsilon > 0$, we have $\lim_{\alpha \rightarrow \infty} \alpha \Psi(x_{\varepsilon,\alpha},y_{\varepsilon,\alpha}) = 0$. For any limit point $(x_\varepsilon,y_\varepsilon)$ of $(x_{\varepsilon,\alpha},y_{\varepsilon,\alpha})$, we have $\Psi(x_{\varepsilon,\alpha},y_{\varepsilon,\alpha}) = 0$.
			\item \label{item:def:continuity_estimate:3} We have 
			for all $\varepsilon > 0$
			\begin{align} 
				& \sup_{\alpha} \cG\left(y_{\varepsilon,\alpha}, - \alpha (\nabla \Psi(x_{\varepsilon,\alpha},\cdot))(y_{\varepsilon,\alpha}),\theta_{\varepsilon,\alpha}\right) < \infty, \label{eqn:control_on_Gbasic_sup} \\
				& \inf_\alpha \cG\left(x_{\varepsilon,\alpha}, \alpha (\nabla \Psi(\cdot,y_{\varepsilon,\alpha}))(x_{\varepsilon,\alpha}),\theta_{\varepsilon,\alpha}\right) > - \infty. \label{eqn:control_on_Gbasic_inf} 	
			\end{align} \label{itemize:funamental_inequality_control_upper_bound}
		\end{enumerate}
		We say that $\cG$ satisfies the \textit{continuity estimate} if for every fundamental collection of variables we have for each $\varepsilon > 0$ that
		\begin{multline}\label{equation:Xi_negative_liminf}
			\liminf_{\alpha \rightarrow \infty} \cG\left(x_{\varepsilon,\alpha}, \alpha (\nabla \Psi(\cdot,y_{\varepsilon,\alpha}))(x_{\varepsilon,\alpha}),\theta_{\varepsilon,\alpha}\right) \\
			- \cG\left(y_{\varepsilon,\alpha}, - \alpha (\nabla \Psi(x_{\varepsilon,\alpha},\cdot))(y_{\varepsilon,\alpha}),\theta_{\varepsilon,\alpha}\right) \leq 0.
		\end{multline}
	\end{definition}
	\begin{remark}
		In Appendix \ref{section:continuity_estimate_general}, we state a slightly more general continuity estimate on the basis of two penalization functions. A proof of a comparison principle on the basis of two penalization functions was given in \cite{Kr17}.
	\end{remark}
	The continuity estimate is indeed exactly the estimate that one would perform when proving the comparison principle for the Hamilton-Jacobi equation in terms of the internal Hamiltonian (disregarding the control $\theta$). 
	Typically, the control on $(x_{\varepsilon,\alpha},y_{\varepsilon,\alpha})$ that is assumed in \ref{item:def:continuity_estimate:1} and \ref{item:def:continuity_estimate:2} is obtained from choosing $(x_{\varepsilon,\alpha},y_{\varepsilon,\alpha})$ as optimizers in the doubling of variables procedure (see Lemma~\ref{lemma:doubling_lemma}), and the control that is assumed in~\ref{item:def:continuity_estimate:3} is obtained by using the viscosity sub- and supersolution properties in the proof of the comparison principle. The required restriction to compact sets in Lemma~\ref{lemma:doubling_lemma} is obtained by including in the test functions a \emph{containment function}.
	
	\begin{definition}[Containment function]\label{def:results:compact-containment}
		We say that a function $\Upsilon : E \rightarrow [0,\infty]$ is a \textit{containment function} for $\Lambda$ if $\Upsilon \in C^1(E)$ and there is a constant $c_\Upsilon$ such that
		\begin{itemize}
			\item For every $c \geq 0$, the set $\{x \, | \, \Upsilon(x) \leq c\}$ is compact;
			\item We have $\sup_\theta\sup_x \Lambda\left(x,\nabla \Upsilon(x),\theta\right) \leq c_\Upsilon$.
		\end{itemize}	
	\end{definition}
	To conclude, our assumption on $\Lambda$ contains the continuity estimate, the controlled growth, the existence of a containment function and two regularity properties.
	\begin{assumption}\label{assumption:results:regularity_of_V}
		The function $\Lambda:E\times\mathbb{R}^d\times\Theta\to\mathbb{R}$ in the Hamiltonian~\eqref{eq:results:variational_hamiltonian} satisfies the following.
		\begin{enumerate}[label=($\Lambda$\arabic*)]
			\item \label{item:assumption:slow_regularity:continuity} The map $\Lambda : E\times\mathbb{R}^d\times\Theta\to\mathbb{R}$ is continuous.
			\item \label{item:assumption:slow_regularity:convexity} For any $x\in E$ and $\theta\in\Theta$, the map $p\mapsto \Lambda(x,p,\theta)$ is convex. We have $\Lambda(x,0,\theta) = 0$ for all $x\in E$ and all $\theta \in \Theta$.
			\item \label{item:assumption:slow_regularity:compact_containment} There exists a containment function $\Upsilon : E \to [0,\infty)$ for $\Lambda$ in the sense of Definition~\ref{def:results:compact-containment}.
			\item \label{item:assumption:slow_regularity:controlled_growth} 
			For every compact set $K \subseteq E$, there exist constants $M, C_1, C_2 \geq 0$  such that for all $x \in K$, $p \in \mathbb{R}^d$ and all $\theta_1,\theta_2\in\Theta$, we have 
			\begin{equation*}
				\Lambda(x,p,\theta_1) \leq \max\left\{M,C_1 \Lambda(x,p,\theta_2) + C_2\right\}.
			\end{equation*}
			\item \label{item:assumption:slow_regularity:continuity_estimate} The function $\Lambda$ satisfies the continuity estimate in the sense of Definition~\ref{def:results:continuity_estimate}, or in the extended sense of Definition~\ref{def:fundamental_inequality_extended}.
		\end{enumerate} 
	\end{assumption}
	Our second main assumption is on the properties of $\cI$. For a compact set~$K\subseteq E$ and a constant~$M\geq 0$, write
	\begin{equation} \label{eqn:def:sublevelsets_I}
		\Theta_{K,M}:= \bigcup_{x \in K} \left\{\theta\in\Theta \, \middle| \,  \mathcal{I}(x,\theta) \leq M \right\},
	\end{equation}
	and
	\begin{equation}
		\Omega_{K,M}:= \bigcap_{x \in K} \left\{\theta\in\Theta \, \middle| \,  \mathcal{I}(x,\theta) \leq M \right\}.
	\end{equation}
	\begin{assumption}\label{assumption:results:regularity_I}
		The functional $\mathcal{I}:E\times\Theta \to [0,\infty]$ in~\eqref{eq:results:variational_hamiltonian} satisfies the following.
		\begin{enumerate}[label=($\mathcal{I}$\arabic*)]
			\item \label{item:assumption:I:lsc} The map $(x,\theta) \mapsto \mathcal{I}(x,\theta)$ is lower semi-continuous on $E \times \Theta$.
			\item \label{item:assumption:I:zero-measure} For any $x\in E$, there exists a control $\theta_{x}^0 \in\Theta$ such that $\mathcal{I}(x,\theta_{x}^0) = 0$. 
			\item \label{item:assumption:I:compact-sublevelsets} For any compact set $K \subseteq E$ and constant $M \geq 0$ the set $\Theta_{K,M}$ is compact.
			\item \label{item:assumption:I:finiteness}
			For each $x \in E$ and constant $M \geq 0$, there exists an open neighbourhood $U \subseteq E$ of $x$ and constants $M',C_1',C_2' \geq 0$ such that for all $y_1,y_2 \in U$ and $\theta \in \Theta_{\{x\},M}$ we have
			\begin{equation*}
				\cI(y_1,\theta) \leq \max \left\{M', C_1'\cI(y_2,\theta) + C_2' \right\}.
			\end{equation*}
			\item \label{item:assumption:I:equi-cont} For every compact set $K \subseteq E$ and each $M \geq 0$ the collection of functions $\{\cI(\cdot,\theta)\}_{\theta \in \Omega_{K,M}}$
			is equicontinuous. That is: for all $\varepsilon > 0$, there is a $\delta > 0$ such that for all $\theta \in \Omega_{K,M}$ and $x,y \in K$ such that $d(x,y) \leq \delta$ we have $|\mathcal{I}(x,\theta) - \mathcal{I}(y,\theta)| \leq \varepsilon$.
		\end{enumerate}
	\end{assumption}

	To establish the existence of viscosity solutions, we will impose one additional assumption. For a general convex functional $p \mapsto \Phi(p)$ we denote
	\begin{multline} \label{eqn:subdifferential}
		\partial_p \Phi(p_0)
		:= \left\{
		\xi \in \mathbb{R}^d \,:\, \Phi(p) \geq \Phi(p_0) + \xi \cdot (p-p_0) \quad (\forall p \in \mathbb{R}^d)
		\right\}.
	\end{multline}

	\begin{definition} \label{definition:tangent_cone}
		The tangent cone (sometimes also called \textit{Bouligand cotingent cone}) to $E$ in $\bR^d$ at $x$ is
		\begin{equation*}
			T_E(x) := \left\{z \in \bR^d \, \middle| \, \liminf_{\lambda \downarrow 0} \frac{d(x + \lambda z, E)}{\lambda} = 0\right\}.
		\end{equation*}
	\end{definition}

	\begin{assumption} \label{assumption:Hamiltonian_vector_field} 
		The set
		$E$ is closed and convex. The map $\Lambda$ is such that $\partial_p \Lambda(x,p,\theta) \subseteq T_E(x)$ for all $x \in E$, $p \in \bR^d$ and $\theta \in \Theta$.
	\end{assumption}
	
	In Lemma \ref{lemma:Hamiltonian_vector_flow_from_Lambda} we will show that the assumption implies that $\partial_p \cH(x,p) \subseteq T_E(x)$, which in turn implies that the solutions of the differential inclusion in terms of $\partial_p \cH(x,p)$ remain inside $E$. 
	Motivated by our examples, we work with closed convex domains~$E$. While in this context we can apply results from e.g. Deimling~\cite{De92}, we believe that similar results can be obtained in different contexts.
	
	\begin{remark} \label{remark:vector_field_inside_relation_comparison}
		The statement that $\partial_p \cH(x,p) \subseteq T_E(x)$ is intuitively implied by the comparison principle for $\bfH$ and therefore, we expect it to hold in any setting for which Theorem \ref{theorem:comparison_principle_variational} holds. Here, we argue in a simple case why this is to be expected. First of all, note that the comparison principle for $\bfH$ builds upon the maximum principle. 
		Suppose that $E = [0,1]$, $f,g \in C^1_b(E)$ and suppose that $f(0) - g(0) = \sup_x f(x) - g(x)$. As $x=0$ is a boundary point, we conclude that $f'(0) \leq g'(0)$. If indeed the maximum principle holds, we must have
		\begin{equation*}
			\cH(0,f'(0)) = Hf(0) \leq Hg(0) = \cH(0,g'(0))
		\end{equation*}
		implying that $p \mapsto \cH(0,p)$ is increasing, in other words 
		\begin{equation*}
			\partial_p \cH(x,p)) \subseteq [0,\infty) = T_{[0,1]}(0).
		\end{equation*}
	\end{remark}
	\section{The comparison principle} \label{section:comparison_principle}
	In this section, we establish Theorem \ref{theorem:comparison_principle_variational}. To establish the comparison principle for $f - \lambda \bfH f = h$ we use the bootstrap method explained in the introduction. We start by a classical localization argument. 
	
	\smallskip
	
	We carry out the localization argument by absorbing the containment function $\Upsilon$ from Assumption \ref{assumption:results:regularity_of_V} \ref{item:assumption:slow_regularity:compact_containment} into the test functions. This leads to two new operators, $H_\dagger$ and $H_\ddagger$ that serve as an upper bound and a lower bound for the true $\bfH$. We will then show the comparison principle for the Hamilton-Jacobi equation in terms of these two new operators.
	We therefore have to extend our notion of Hamilton-Jacobi equations and the comparison principle. This extension of the definition is standard, but we included it for completeness in the appendix as Definition \ref{definition:appendix_pair_ofHJ}.
	
	\smallskip
	
	This procedure allows us to clearly separate the reduction to compact sets on one hand, and the proof of the comparison principle on the basis of the bootstrap procedure on the other. Schematically, we will establish the following diagram:
	
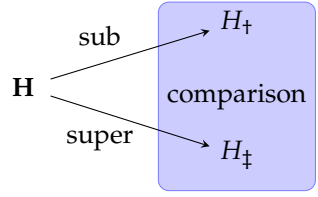
\begin{center}
		\begin{tikzpicture}
			\matrix (m) [matrix of math nodes,row sep=1em,column sep=4em,minimum width=2em]
			{
				{ } &[7mm] H_\dagger \\
				\bfH & { } \\
				{ }  & H_\ddagger \\};
			\path[-stealth]
			(m-2-1) edge node [above] {sub \qquad { }} (m-1-2)
			(m-2-1) edge node [below] {super \qquad { }} (m-3-2);
			
			\begin{pgfonlayer}{background}
				\node at (m-2-2) [rectangle,draw=blue!50,fill=blue!20,rounded corners, minimum width=1cm, minimum height=2.5cm]  {comparison};
			\end{pgfonlayer}
		\end{tikzpicture}
	\end{center}
	In this diagram, an arrow connecting an operator $A$ with operator $B$ with subscript 'sub' means that viscosity subsolutions of $f - \lambda A f = h$ are also viscosity subsolutions of $f - \lambda B f = h$. Similarly for arrows with a subscript 'super'.
	\smallskip
	
	We introduce the operators $H_\dagger$ and $H_\ddagger$ in Section~\ref{subsection:definition_of_Hamiltonians}. The arrows will be established in Section \ref{subsection:implications_from_compact_containment}. Finally, we will establish the comparison principle for $H_\dagger$ and $H_\ddagger$ in Section~\ref{subsection:proof_of_comparison_principle}. Combined these two results imply the comparison principle for $\bfH$.
	\begin{proof}[Proof of Theorem~\ref{theorem:comparison_principle_variational}]
		We start with the proof of (a). Let $f \in \cD(\bfH)$. Then $\bfH f$ is continuous since by Proposition~\ref{prop:reg-of-H-and-L:continuity} in Appendix \ref{section:regularity-of-H-and-L}, the Hamiltonian~$\mathcal{H}$ is continuous.
		
		\smallskip

		We proceed with the proof of (b). Fix $h_1,h_2 \in C_b(E)$ and $\lambda > 0$.
		
		\smallskip 
		
		Let $u_1,u_2$ be a viscosity sub- and supersolution to $f - \lambda \bfH f = h_1$ and  $f - \lambda \bfH f = h_2$ respectively. By Lemma \ref{lemma:viscosity_solutions_compactify2} proven in Section~\ref{subsection:implications_from_compact_containment}, $u_1$ and $u_2$ are a sub- and supersolution to $f - \lambda H_\dagger f = h_1$ and $f - \lambda H_\ddagger f = h_2$ respectively. Thus $\sup_E u_1 - u_2 \leq \sup_E h_1 - h_2$ by Proposition~\ref{prop:CP} of Section~\ref{subsection:proof_of_comparison_principle}. Specialising to $h_1=h_2$ gives Theorem~\ref{theorem:comparison_principle_variational}.
	\end{proof}
	\subsection{Definition of auxiliary operators} \label{subsection:definition_of_Hamiltonians}
	In this section, we repeat the definition of $\bfH$, and introduce the operators $H_\dagger$ and $H_\ddagger$.
	\begin{definition} \label{definition_effectiveH}
		The operator $\bfH \subseteq C_b^1(E) \times C_b(E)$ has domain $\cD(\bfH) = C_{cc}^\infty(E)$ and satisfies $\bfH f(x) = \cH(x, \nabla f(x))$, where $\cH$ is the map
		\begin{equation*}
			\mathcal{H}(x,p) = \sup_{\theta \in \Theta}\left[\Lambda(x,p,\theta) - \mathcal{I}(x,\theta)\right].
		\end{equation*}
	\end{definition}
	We proceed by introducing $H_\dagger$ and $H_\ddagger$. Recall  Assumption~\ref{item:assumption:slow_regularity:compact_containment} and the constant $C_\Upsilon := \sup_{\theta}\sup_x \Lambda(x,\nabla \Upsilon(x),\theta)$ therein.
	Denote by $C_\ell^\infty(E)$ the set of smooth functions on $E$ that have a lower bound and by $C_u^\infty(E)$ the set of smooth functions on $E$ that have an upper bound.
	\begin{definition}[The operators $H_\dagger$ and $H_\ddagger$] \label{definiton:HdaggerHddagger}
		For $f \in C_\ell^\infty(E)$ and $\varepsilon \in (0,1)$  set 
		\begin{gather*}
			f^\varepsilon_\dagger := (1-\varepsilon) f + \varepsilon \Upsilon \\
			H_{\dagger,f}^\varepsilon(x) := (1-\varepsilon) \cH(x,\nabla f(x)) + \varepsilon C_\Upsilon.
		\end{gather*}
		and set
		\begin{equation*}
			H_\dagger := \left\{(f^\varepsilon_\dagger,H_{\dagger,f}^\varepsilon) \, \middle| \, f \in C_\ell^\infty(E), \varepsilon \in (0,1) \right\}.
		\end{equation*} 
		For $f \in C_u^\infty(E)$ and $\varepsilon \in (0,1)$  set 
		\begin{gather*}
			f^\varepsilon_\ddagger := (1+\varepsilon) f - \varepsilon \Upsilon \\
			H_{\ddagger,f}^\varepsilon(x) := (1+\varepsilon) \cH(x,\nabla f(x)) - \varepsilon C_\Upsilon.
		\end{gather*}
		and set
		\begin{equation*}
			H_\ddagger := \left\{(f^\varepsilon_\ddagger,H_{\ddagger,f}^\varepsilon) \, \middle| \, f \in C_u^\infty(E), \varepsilon \in (0,1) \right\}.
		\end{equation*} 
	\end{definition}
	
	\subsection{Preliminary results} \label{subsection:implications_from_compact_containment}

	The operator $\bfH$ is related to $H_\dagger, H_\ddagger$ by the following Lemma.
	\begin{lemma}\label{lemma:viscosity_solutions_compactify2}
		Fix $\lambda > 0$ and $h \in C_b(E)$. 
		\begin{enumerate}[(a)]
			\item Every subsolution to $f - \lambda \bfH f = h$ is also a subsolution to $f - \lambda H_\dagger f = h$.
			\item Every supersolution to $f - \lambda \bfH f = h$ is also a supersolution to~$f-\lambda H_\ddagger f=~h$.
		\end{enumerate}
	\end{lemma}
	We only prove (a) of Lemma~\ref{lemma:viscosity_solutions_compactify2}, as (b) can be carried out analogously.

	\begin{proof}
		Fix $\lambda > 0$ and $h \in C_b(E)$. Let $u$ be a subsolution to $f - \lambda \mathbf{H}f = h$. We prove it is also a subsolution to $f - \lambda H_\dagger f = h$.
		\smallskip
		
		Fix $\varepsilon > 0 $ and $f\in C_\ell^\infty(E)$ and let $(f^\varepsilon_\dagger,H^\varepsilon_{\dagger,f}) \in H_\dagger$ as in Definition \ref{definiton:HdaggerHddagger}. We will prove that there are $x_n\in E$ such that
		\begin{gather}
			\lim_{n\to\infty}\left(u-f_\dagger^\varepsilon\right)(x_n) = \sup_{x\in E}\left(u(x)-f_\dagger^\varepsilon(x) \right),\label{eqn:proof_lemma_conditions_for_subsolution_first}\\
			\limsup_{n\to\infty} \left[u(x_n)-\lambda H_{\dagger,f}^\varepsilon(x_n) - h(x_n)\right]\leq 0.\label{eqn:proof_lemma_conditions_for_subsolution_second}
		\end{gather}
		As the function $\left[u -(1-\varepsilon)f\right]$ is bounded from above and $\varepsilon \Upsilon$ has compact sublevel-sets, the sequence $x_n$ along which the first limit is attained can be assumed to lie in the compact set 
		\begin{equation*}
			K := \left\{x \, | \, \Upsilon(x) \leq \varepsilon^{-1} \sup_x \left(u(x) - (1-\varepsilon)f(x) \right)\right\}.
		\end{equation*}
		Set $M = \varepsilon^{-1} \sup_x \left(u(x) - (1-\varepsilon)f(x) \right)$. Let $\gamma : \bR \rightarrow \bR$ be a smooth increasing function such that
		\begin{equation*}
			\gamma(r) = \begin{cases}
				r & \text{if } r \leq M, \\
				M + 1 & \text{if } r \geq M+2.
			\end{cases}
		\end{equation*}
		Denote by $f_\varepsilon$ the function on $E$ defined by 
		\begin{equation*}
			f_\varepsilon(x) := \gamma\left((1-\varepsilon)f(x) + \varepsilon \Upsilon(x) \right).
		\end{equation*}
		By construction $f_\varepsilon$ is smooth and constant outside of a compact set and thus lies in $\cD(H) = C_{cc}^\infty(E)$. As $u$ is a viscosity subsolution for $f - \lambda Hf = h$ there exists a sequence $x_n \in K \subseteq E$ (by our choice of $K$) with
		\begin{gather}
			\lim_n \left(u-f_\varepsilon\right)(x_n) = \sup_x \left(u(x)-f_\varepsilon(x)\right), \label{eqn:visc_subsol_sup} \\
			\limsup_n \left[u(x_n) - \lambda \mathbf{H} f_\varepsilon(x_n) - h(x_n)\right] \leq 0. \label{eqn:visc_subsol_upperbound}
		\end{gather}
		As $f_\varepsilon$ equals $f_\dagger^\varepsilon$ on $K$, we have from \eqref{eqn:visc_subsol_sup} that also
		\begin{equation*}
			\lim_n \left(u-f_\dagger^\varepsilon\right)(x_n) = \sup_{x\in E}\left(u(x)-f_\dagger^\varepsilon(x)\right),
		\end{equation*}
		establishing~\eqref{eqn:proof_lemma_conditions_for_subsolution_first}. Convexity of $p \mapsto \mathcal{H}(x,p)$ yields for arbitrary points $x\in K$ the estimate
		\begin{align*}
			\mathbf{H} f_\varepsilon(x) &= \mathcal{H}(x,\nabla f_\varepsilon(x)) \\
			& \leq (1-\varepsilon) \mathcal{H}(x,\nabla f(x)) + \varepsilon \mathcal{H}(x,\nabla \Upsilon(x)) \\
			&\leq (1-\varepsilon) \mathcal{H}(x,\nabla f(x)) + \varepsilon C_\Upsilon = H^\varepsilon_{\dagger,f}(x).
		\end{align*} 
		Combining this inequality with \eqref{eqn:visc_subsol_upperbound} yields
		\begin{multline*}
			\limsup_n \left[u(x_n) - \lambda H^\varepsilon_{\dagger,f}(x_n) - h(x_n)\right] \\
			\leq \limsup_n \left[u(x_n) - \lambda \mathbf{H} f_\varepsilon(x_n) - h(x_n)\right] \leq 0,
		\end{multline*}
		establishing \eqref{eqn:proof_lemma_conditions_for_subsolution_second}. This concludes the proof.
	\end{proof}
	\subsection{The comparison principle} \label{subsection:proof_of_comparison_principle}
	In this section, we prove the comparison principle for the operators $H_\dagger$ and $H_\ddagger$.
	\begin{proposition}\label{prop:CP} 
		Fix $\lambda > 0$ and $h_1,h_2 \in C_b(E)$. 	Let $u_1$ be a viscosity subsolution to $f - \lambda H_\dagger f = h_1$ and let $u_2$ be a viscosity supersolution to $f - \lambda H_\ddagger f = h_2$. Then we have $\sup_x u_1(x) - u_2(x) \leq \sup_x h_1(x) - h_2(x)$.
	\end{proposition}

	The proof uses a variant of a classical estimate that was proven e.g. in \cite[Proposition 3.7]{CIL92} or in the present form in Proposition~A.11 of \cite{CoKr17}. 
	
	\begin{lemma}\label{lemma:doubling_lemma}
		Let $u$ be bounded and upper semi-continuous, let $v$ be bounded and lower semi-continuous, let $\Psi : E^2 \rightarrow \bR^+$ be penalization functions and let $\Upsilon$ be a containment function.
		\smallskip
		
		Fix $\varepsilon > 0$. For every $\alpha >0$ there exist $x_{\alpha,\varepsilon},y_{\alpha,\varepsilon} \in E$ such that
		\begin{multline} \label{eqn:existence_optimizers}
			\frac{u(x_{\alpha,\varepsilon})}{1-\varepsilon} - \frac{v(y_{\alpha,\varepsilon})}{1+\varepsilon} - \alpha \Psi(x_{\alpha,\varepsilon},y_{\alpha,\varepsilon}) - \frac{\varepsilon}{1-\varepsilon}\Upsilon(x_{\alpha,\varepsilon}) -\frac{\varepsilon}{1+\varepsilon}\Upsilon(y_{\alpha,\varepsilon}) \\
			= \sup_{x,y \in E} \left\{\frac{u(x)}{1-\varepsilon} - \frac{v(y)}{1+\varepsilon} -  \alpha \Psi(x,y)  - \frac{\varepsilon}{1-\varepsilon}\Upsilon(x) - \frac{\varepsilon}{1+\varepsilon}\Upsilon(y)\right\}.
		\end{multline}
		Additionally, for every $\varepsilon > 0$ we have that
		\begin{enumerate}[(a)]
			\item The set $\{x_{\alpha,\varepsilon}, y_{\alpha,\varepsilon} \, | \,  \alpha > 0\}$ is relatively compact in $E$.
			\item All limit points of $\{(x_{\alpha,\varepsilon},y_{\alpha,\varepsilon})\}_{\alpha > 0}$ as $\alpha \rightarrow \infty$ are of the form $(z,z)$ and for these limit points we have $u(z) - v(z) = \sup_{x \in E} \left\{u(x) - v(x) \right\}$.
			\item We have 
			\[
			\lim_{\alpha \rightarrow \infty}  \alpha \Psi(x_{\alpha,\varepsilon},y_{\alpha,\varepsilon}) = 0.
			\]
		\end{enumerate}
	\end{lemma}
	\begin{proof}[Proof of Proposition~\ref{prop:CP}]
		Fix $\lambda >0$ and $h_1,h_2 \in C_b(E)$. Let $u_1$ be a viscosity subsolution and $u_2$ be a viscosity supersolution of $f - \lambda H_\dagger f = h_1$ and  $f - \lambda H_\ddagger f = h_2$ respectively.  We prove Theorem~\ref{prop:CP} in five steps of which the first two are classical.  
		
		We sketch the steps, before giving full proofs.
		
		\smallskip
		
		\underline{\emph{Step 1}}:  We prove that for $\varepsilon > 0 $ and $\alpha > 0$, there exist points $x_{\varepsilon,\alpha},y_{\varepsilon,\alpha} \in E$ satisfying the properties listed in Lemma \ref{lemma:doubling_lemma} and momenta $p_{\varepsilon,\alpha}^1,p_{\varepsilon,\alpha}^2 \in \mathbb{R}^d$ such that
		\begin{equation*}
			p_{\varepsilon,\alpha}^1 = \alpha \nabla \Psi(\cdot,y_{\varepsilon,\alpha})(x_{\varepsilon,\alpha}), \qquad p_{\varepsilon,\alpha}^2 = - \alpha \nabla \Psi(x_{\varepsilon,\alpha},\cdot)(y_{\varepsilon,\alpha}),
		\end{equation*}
		and
		\begin{multline} \label{eqn:estimate_step_1}
			\sup_E(u_1-u_2) \leq \lambda \liminf_{\varepsilon\to 0}\liminf_{\alpha \to \infty} \left[\mathcal{H}(x_{\varepsilon,\alpha},p^1_{\varepsilon,\alpha}) - \mathcal{H}(y_{\varepsilon,\alpha},p^2_{\varepsilon,\alpha})\right] \\ + \sup_{E}(h_1 - h_2).
		\end{multline}
		This step is solely based on the sub- and supersolution properties of $u_1,u_2$, the continuous differentiability of the penalization function $\Psi(x,y)$, the containment function $\Upsilon$, and convexity of $p \mapsto \mathcal{H}(x,p)$. We conclude it suffices to establish for each $\varepsilon > 0$ that
		\begin{equation} \label{eqn:difference_hamiltonians_in_comparison}
			\liminf_{\alpha \rightarrow \infty} \mathcal{H}(x_{\varepsilon,\alpha},p^1_{\varepsilon,\alpha}) - \mathcal{H}(y_{\varepsilon,\alpha},p^2_{\varepsilon,\alpha}) \leq 0.
		\end{equation}

		\underline{\emph{Step 2}}: We will show that there are controls $\theta_{\varepsilon,\alpha}$ such that
		\begin{equation} \label{eqn:choice_control}
			\mathcal{H}(x_{\varepsilon,\alpha},p^1_{\varepsilon,\alpha}) = \Lambda(x_{\varepsilon,\alpha},p^1_{\varepsilon,\alpha},\theta_{\varepsilon,\alpha}) - \mathcal{I}(x_{\varepsilon,\alpha},\theta_{\varepsilon,\alpha}).
		\end{equation}
		As a consequence we have
		\begin{multline} \label{eqn:basic_decomposition_Hamiltonian_difference1}
			\mathcal{H}(x_{\varepsilon,\alpha},p^1_{\varepsilon,\alpha})-
			\mathcal{H}(y_{\varepsilon,\alpha},p^2_{\varepsilon,\alpha})
			\leq 
			\Lambda(x_{\varepsilon,\alpha},p^1_{\varepsilon,\alpha},\theta_{\varepsilon,\alpha})-
			\Lambda(y_{\varepsilon,\alpha},p^2_{\varepsilon,\alpha},\theta_{\varepsilon,\alpha})\\
			+\mathcal{I}(y_{\varepsilon,\alpha},\theta_{\varepsilon,\alpha})-
			\mathcal{I}(x_{\varepsilon,\alpha},\theta_{\varepsilon,\alpha}).
		\end{multline}
		For establishing \eqref{eqn:difference_hamiltonians_in_comparison}, it is sufficient to bound the differences in \eqref{eqn:basic_decomposition_Hamiltonian_difference1} by using Assumptions \ref{assumption:results:regularity_of_V} \ref{item:assumption:slow_regularity:continuity_estimate}  and  \ref{assumption:results:regularity_I} \ref{item:assumption:I:equi-cont}.
		
		\underline{\emph{Step 3}}: We verify the conditions to apply the continuity estimate, Assumption \ref{assumption:results:regularity_of_V} \ref{item:assumption:slow_regularity:continuity_estimate}.
		
		\smallskip
		
		The bootstrap argument allows us to find for each $\varepsilon$ a subsequence $\alpha = \alpha(\varepsilon) \rightarrow \infty$ such that the variables $(x_{\varepsilon,\alpha},x_{\varepsilon,\alpha},\theta_{\varepsilon,\alpha})$ are fundamental for $\Lambda$ with respect to $\Psi$ (See Definition \ref{def:results:continuity_estimate}).

		\underline{\emph{Step 4}}: We verify the conditions to apply the estimate on $\cI$, Assumption \ref{assumption:results:regularity_I} \ref{item:assumption:I:equi-cont}.

		\smallskip
		
		\underline{\emph{Step 5}}: Using the outcomes of Steps 3 and 4, we can apply the continuity estimate of Assumption \ref{assumption:results:regularity_of_V} \ref{item:assumption:slow_regularity:controlled_growth} and the equi-continuity of Assumption \ref{assumption:results:regularity_I} \ref{item:assumption:I:equi-cont} to estimate \eqref{eqn:basic_decomposition_Hamiltonian_difference1} for any $\varepsilon$:
		\begin{multline} \label{eqn:basic_decomposition_Hamitlonian_difference2}
			\liminf_{\alpha \rightarrow \infty} \mathcal{H}(x_{\varepsilon,\alpha},p^1_{\varepsilon,\alpha})-
			\mathcal{H}(y_{\varepsilon,\alpha},p^2_{\varepsilon,\alpha}) \\
			\leq 
			\liminf_{\alpha \rightarrow \infty} \Lambda(x_{\varepsilon,\alpha},p^1_{\varepsilon,\alpha},\theta_{\varepsilon,\alpha})-
			\Lambda(y_{\varepsilon,\alpha},p^2_{\varepsilon,\alpha},\theta_{\varepsilon,\alpha}) \\
			+\mathcal{I}(y_{\varepsilon,\alpha},\theta_{\varepsilon,\alpha})-
			\mathcal{I}(x_{\varepsilon,\alpha},\theta_{\varepsilon,\alpha}) \leq 0,
		\end{multline}
		which establishes \eqref{eqn:difference_hamiltonians_in_comparison} and thus also the comparison principle.
		
		\smallskip
		
		We proceed with the proofs of the first four steps, as the fifth step is immediate.
		\smallskip
		
		\underline{\emph{Proof of Step 1}}: The proof of this first step is classical. We include it for completeness.
		For any $\varepsilon > 0$ and any $\alpha > 0$, define the map $\Phi_{\varepsilon,\alpha}: E \times E \to \mathbb{R}$ by
		\begin{equation*}
			\Phi_{\varepsilon,\alpha}(x,y) := \frac{u_1(x)}{1-\varepsilon} - \frac{u_2(y)}{1+\varepsilon} - \alpha \Psi(x,y) - \frac{\varepsilon}{1-\varepsilon} \Upsilon(x) - \frac{\varepsilon}{1+\varepsilon}\Upsilon(y).
		\end{equation*}
		Let $\varepsilon > 0$. By Lemma \ref{lemma:doubling_lemma}, there is a compact set $K_\varepsilon \subseteq E$ and there exist points $x_{\varepsilon,\alpha},y_{\varepsilon,\alpha} \in K_\varepsilon$ such that
		\begin{equation} \label{eqn:comparison_optimizers}
			\Phi_{\varepsilon,\alpha}(x_{\varepsilon,\alpha},y_{\varepsilon,\alpha}) = \sup_{x,y \in E} \Phi_{\varepsilon,\alpha}(x,y),
		\end{equation}
		and 
		\begin{equation}\label{eq:proof-CP:Psi-xy-converge}
			\lim_{\alpha \to \infty} \alpha \Psi(x_{\varepsilon,\alpha},y_{\varepsilon,\alpha}) = 0.
		\end{equation}
		As in the proof of Proposition~A.11 of~\cite{Kr17}, it follows that
		\begin{equation}\label{eq:proof-CP:general-bound-u1u2}
			\sup_E (u_1 - u_2) \leq \liminf_{\varepsilon \to 0} \liminf_{\alpha \to \infty} \left[ \frac{u_1(x_{\varepsilon,\alpha})}{1-\varepsilon} - \frac{u_2(y_{\varepsilon,\alpha})}{1+\varepsilon}\right].
		\end{equation}
		At this point, we want to use the sub- and supersolution properties of $u_1$ and $u_2$. Define the test functions $\varphi^{\varepsilon,\alpha}_1 \in \cD(H_\dagger), \varphi^{\varepsilon,\alpha}_2 \in \cD(H_\ddagger)$ by
		\begin{align*}
			\varphi^{\varepsilon,\alpha}_1(x) & := (1-\varepsilon) \left[\frac{u_2(y_{\varepsilon,\alpha})}{1+\varepsilon} + \alpha \Psi(x,y_{\varepsilon,\alpha}) + \frac{\varepsilon}{1-\varepsilon}\Upsilon(x) + \frac{\varepsilon}{1+\varepsilon}\Upsilon(y_{\varepsilon,\alpha})\right] \\
			& \hspace{5cm} + (1-\varepsilon)(x-x_{\varepsilon,\alpha})^2, \\
			\varphi^{\varepsilon,\alpha}_2(y) & := (1+\varepsilon)\left[\frac{u_1(x_{\varepsilon,\alpha})}{1-\varepsilon} - \alpha \Psi(x_{\varepsilon,\alpha},y) - \frac{\varepsilon}{1-\varepsilon}\Upsilon(x_{\varepsilon,\alpha}) - \frac{\varepsilon}{1+\varepsilon}\Upsilon(y)\right] \\
			& \hspace{5cm}  - (1+\varepsilon) (y-y_{\varepsilon,\alpha})^2.
		\end{align*}
		Using \eqref{eqn:comparison_optimizers}, we find that $u_1 - \varphi^{\varepsilon,\alpha}_1$ attains its supremum at $x = x_{\varepsilon,\alpha}$, and thus
		\begin{equation*}
			\sup_E (u_1-\varphi^{\varepsilon,\alpha}_1) = (u_1-\varphi^{\varepsilon,\alpha}_1)(x_{\varepsilon,\alpha}).
		\end{equation*}
		Denote $p_{\varepsilon,\alpha}^1 := \alpha \nabla_x \Psi(x_{\varepsilon,\alpha},y_{\varepsilon,\alpha})$. By our addition of the penalization $(x-x_{\varepsilon,\alpha})^2$ to the test function, the point $x_{\varepsilon,\alpha}$ is in fact the unique optimizer, and we obtain from the subsolution inequality that
		\begin{equation}\label{eq:proof-CP:subsol-ineq}
			u_1(x_{\varepsilon,\alpha}) - \lambda \left[ (1-\varepsilon) \mathcal{H}\left(x_{\varepsilon,\alpha}, p_{\varepsilon,\alpha}^1 \right) + \varepsilon C_\Upsilon\right] \leq h_1(x_{\varepsilon,\alpha}).
		\end{equation}	
		With a similar argument for $u_2$ and $\varphi^{\varepsilon,\alpha}_2$, we obtain by the supersolution inequality that
		\begin{equation}\label{eq:proof-CP:supersol-ineq}
			u_2(y_{\varepsilon,\alpha}) - \lambda \left[(1+\varepsilon)\mathcal{H}\left(y_{\varepsilon,\alpha}, p_{\varepsilon,\alpha}^2 \right) - \varepsilon C_\Upsilon\right] \geq h_2(y_{\varepsilon,\alpha}),
		\end{equation}
		where $p_{\varepsilon,\alpha}^2 := -\alpha \nabla_y \Psi(x_{\varepsilon,\alpha},y_{\varepsilon,\alpha})$. With that, estimating further in~\eqref{eq:proof-CP:general-bound-u1u2} leads to
		\begin{multline*}
			\sup_E(u_1-u_2) \leq \liminf_{\varepsilon\to 0}\liminf_{\alpha \to \infty} \bigg[\frac{h_1(x_{\varepsilon,\alpha})}{1-\varepsilon} - \frac{h_2(y_{\varepsilon,\alpha})}{1+\varepsilon} + \frac{\varepsilon}{1-\varepsilon} C_\Upsilon \\ + \frac{\varepsilon}{1+\varepsilon} C_\Upsilon  + \lambda \left[\mathcal{H}(x_{\varepsilon,\alpha},p^1_{\varepsilon,\alpha}) - \mathcal{H}(y_{\varepsilon,\alpha},p^2_{\varepsilon,\alpha})\right]\bigg].
		\end{multline*}
		Thus, \eqref{eqn:estimate_step_1} in Step 1 follows.
		
		\smallskip

		\underline{\emph{Proof of Step 2}}: Recall that $\mathcal{H}(x,p)$ is given by
		\begin{equation*}
			\mathcal{H}(x,p) = \sup_{\theta \in \Theta}\left[\Lambda(x,p,\theta) - \mathcal{I}(x,\theta)\right].
		\end{equation*}
		Since $\Lambda(x_{\varepsilon,\alpha},p^1_{\varepsilon,\alpha},\cdot) : \Theta \to \mathbb{R}$ is bounded and continuous by \ref{item:assumption:slow_regularity:continuity}  and \ref{item:assumption:slow_regularity:controlled_growth}, and $\mathcal{I}(x_{\varepsilon,\alpha},\cdot) : \Theta \to [0,\infty]$ has compact sub-level sets in $\Theta$ by~\ref{item:assumption:I:compact-sublevelsets}, there exists an optimizer $\theta_{\varepsilon,\alpha} \in\Theta$ such that
		\begin{equation} \label{eqn:choice_of_optimal_measure}
			\mathcal{H}(x_{\varepsilon,\alpha},p^1_{\varepsilon,\alpha}) = \Lambda(x_{\varepsilon,\alpha},p^1_{\varepsilon,\alpha},\theta_{\varepsilon,\alpha}) - \mathcal{I}(x_{\varepsilon,\alpha},\theta_{\varepsilon,\alpha}).
		\end{equation}
		Choosing the same point in the supremum of the second term $\mathcal{H}(y_{\varepsilon,\alpha},p^2_{\varepsilon,\alpha})$, we obtain for all $\varepsilon > 0$ and $\alpha > 0$ the estimate
		\begin{multline} \label{eqn:basic_decomposition_Hamiltonian_difference}
			\mathcal{H}(x_{\varepsilon,\alpha},p^1_{\varepsilon,\alpha})-
			\mathcal{H}(y_{\varepsilon,\alpha},p^2_{\varepsilon,\alpha})
			\leq 
			\Lambda(x_{\varepsilon,\alpha},p^1_{\varepsilon,\alpha},\theta_{\varepsilon,\alpha})-
			\Lambda(y_{\varepsilon,\alpha},p^2_{\varepsilon,\alpha},\theta_{\varepsilon,\alpha})\\
			+\mathcal{I}(y_{\varepsilon,\alpha},\theta_{\varepsilon,\alpha})-
			\mathcal{I}(x_{\varepsilon,\alpha},\theta_{\varepsilon,\alpha}).
		\end{multline}
		
		\underline{\emph{Proof of Step 3}}: 
		We will construct for each $\varepsilon > 0$ a sequence $\alpha = \alpha(\varepsilon) \rightarrow \infty$ such that the collection $(x_{\varepsilon,\alpha},y_{\varepsilon,\alpha},\theta_{\varepsilon,\alpha})$ is fundamental for $\Lambda$ with respect to $\Psi$ in the sense of Definition \ref{def:results:continuity_estimate}. We thus need to verify for each $\varepsilon > 0$ 
		\begin{enumerate}[(i)]
			\item \label{item:fundamental_1iminf}
			\begin{equation}
				\inf_\alpha \Lambda(x_{\varepsilon,\alpha},p^1_{\varepsilon,\alpha},\theta_{\varepsilon,\alpha}) > - \infty,\label{eq:proof-CP:Vx-unif-bound-below_first}
			\end{equation}
			\item \label{item:fundamental_1imsup}
			\begin{equation}
				\sup_{\alpha}\Lambda(y_{\varepsilon,\alpha},p^2_{\varepsilon,\alpha},\theta_{\varepsilon,\alpha}) < \infty \label{eq:proof-CP:Vy-unif-bound-above_first}
			\end{equation}
			\item \label{item:fundamental_compactcontrols} The set of controls $\theta_{\varepsilon,\alpha}$ is relatively compact.
		\end{enumerate}
		
		To prove \ref{item:fundamental_1iminf}, \ref{item:fundamental_1imsup} and \ref{item:fundamental_compactcontrols}, we introduce auxiliary controls $\theta_{\varepsilon,\alpha}^0$, obtained by \ref{item:assumption:I:zero-measure}, satisfying
		\begin{equation}\label{eqn:choice_of_zero_measure}
			\mathcal{I}(y_{\varepsilon,\alpha},\theta_{\varepsilon,\alpha}^0) = 0.
		\end{equation}
		We will first establish \ref{item:fundamental_1iminf} and \ref{item:fundamental_1imsup} for all $\alpha$. Then, for the proof of \ref{item:fundamental_compactcontrols}, we will construct for each $\varepsilon > 0$ a suitable subsequence $\alpha \rightarrow \infty$.
		
		\underline{\emph{Proof of Step 3, \ref{item:fundamental_1iminf} and \ref{item:fundamental_1imsup}}}: 
		
		We first establish \ref{item:fundamental_1iminf}. By the subsolution inequality~\eqref{eq:proof-CP:subsol-ineq},
		\begin{align}\label{eq:proof-CP:estimate-via-subsol}
			\frac{1}{\lambda} \inf_E\left(u_1 - h\right) & \leq (1-\varepsilon) \mathcal{H}(x_{\varepsilon,\alpha},p^1_{\varepsilon,\alpha}) + \varepsilon C_{\Upsilon}	\\
			&\leq (1-\varepsilon) \Lambda(x_{\varepsilon,\alpha},p^1_{\varepsilon,\alpha},\theta_{\varepsilon,\alpha}) + \varepsilon C_\Upsilon,\notag
		\end{align}
		and the lower bound~\eqref{eq:proof-CP:Vx-unif-bound-below_first} follows.
		
		\smallskip
		
		We next establish \ref{item:fundamental_1imsup}. By the supersolution inequality~\eqref{eq:proof-CP:supersol-ineq}, we can estimate 
		\begin{align*}
			(1+\varepsilon) \Lambda(y_{\varepsilon,\alpha},p^2_{\varepsilon,\alpha},\theta_{\varepsilon,\alpha}^0) &= (1+\varepsilon) \left[\Lambda(y_{\varepsilon,\alpha},p^2_{\varepsilon,\alpha},\theta_{\varepsilon,\alpha}^0) - \mathcal{I}(y_{\varepsilon,\alpha},\theta_{\varepsilon,\alpha}^0)\right]	\\
			&\leq \left((1+\varepsilon) \mathcal{H}\left(y_{\varepsilon,\alpha},p^2_{\varepsilon,\alpha}\right) - \varepsilon C_\Upsilon\right) + \varepsilon C_\Upsilon	\\
			&\leq \frac{1}{\lambda} \sup_E (u_2-h) + \varepsilon C_{\Upsilon} < \infty,
		\end{align*}
		and the upper bound~\eqref{eq:proof-CP:Vy-unif-bound-above_first} follows by Assumption \ref{assumption:results:regularity_of_V} \ref{item:assumption:slow_regularity:controlled_growth}.

		\smallskip

		\underline{\emph{Proof of Step 3, \ref{item:fundamental_compactcontrols}}}: 
		To prove \ref{item:fundamental_compactcontrols},  it suffices by Assumption \ref{assumption:results:regularity_I} \ref{item:assumption:I:compact-sublevelsets} to find for each $\varepsilon > 0$ a subsequence $\alpha$ such that 
		\begin{equation} \label{eqn:fundamental_controlI}
			\sup_{\alpha} \cI(x_{\varepsilon,\alpha},\theta_{\varepsilon,\alpha}) < \infty.
		\end{equation}
		By \eqref{eq:proof-CP:estimate-via-subsol}, we have
		\begin{align*}
			\frac{1}{\lambda} \inf_E\left(u_1 - h\right) & \leq (1-\varepsilon) \cH(x_{\varepsilon,\alpha},p^1_{\varepsilon,\alpha}) + \varepsilon C_{\Upsilon} \\
			& = (1-\varepsilon) \left[\Lambda(x_{\varepsilon,\alpha},p^1_{\varepsilon,\alpha},\theta_{\varepsilon,\alpha}) - \mathcal{I}(x_{\varepsilon,\alpha},\theta_{\varepsilon,\alpha}) \right] + \varepsilon C_{\Upsilon}.
		\end{align*}
		We conclude that $\sup_\alpha \mathcal{I}(x_{\varepsilon,\alpha},\theta_{\varepsilon,\alpha}) < \infty$ is implied by
		\begin{equation*}
			\sup_{\alpha} \Lambda(x_{\varepsilon,\alpha},p^1_{\varepsilon,\alpha},\theta_{\varepsilon,\alpha}) < \infty
		\end{equation*}
		which by \ref{item:assumption:slow_regularity:controlled_growth} is equivalent to
		\begin{equation} \label{eqn:fundamental_controlLambda}
			\sup_\alpha \Lambda(x_{\varepsilon,\alpha},p^1_{\varepsilon,\alpha},\theta_{\varepsilon,\alpha}^0) < \infty.
		\end{equation}
		To perform this estimate, we first write 
		\begin{multline}
			\Lambda(x_{\varepsilon,\alpha},p^1_{\varepsilon,\alpha},\theta_{\varepsilon,\alpha}^0) \\
			= \Lambda(y_{\varepsilon,\alpha},p^2_{\varepsilon,\alpha},\theta_{\varepsilon,\alpha}^0) + \left[\Lambda(x_{\varepsilon,\alpha},p^1_{\varepsilon,\alpha},\theta_{\varepsilon,\alpha}^0) -  \Lambda(y_{\varepsilon,\alpha},p^2_{\varepsilon,\alpha},\theta_{\varepsilon,\alpha}^0)\right]. \label{eqn:fundamental_bootstrap_split}
		\end{multline}
		To estimate the second term, we aim to apply the continuity estimate for the controls $\theta_{\varepsilon,\alpha}^0$. To do so, must establish that $(x_{\varepsilon,\alpha},y_{\varepsilon,\alpha},\theta_{\varepsilon,\alpha}^0)$ is fundamental for $\Lambda$ with respect to $\Psi$. By Assumption \ref{assumption:results:regularity_I} \ref{item:assumption:I:compact-sublevelsets}, for each $\varepsilon$ the set of  controls $\theta_{\varepsilon,\alpha}^0$ is relatively compact. Thus it suffices to establish
		\begin{align}
			& \inf_\alpha \Lambda(x_{\varepsilon,\alpha},p^1_{\varepsilon,\alpha},\theta_{\varepsilon,\alpha}^0) > - \infty, \label{eq:proof-CP:Vx-unif-bound-below_zero_measure} \\
			& \sup_{\alpha}\Lambda(y_{\varepsilon,\alpha},p^2_{\varepsilon,\alpha},\theta_{\varepsilon,\alpha}^0) < \infty. \label{eq:proof-CP:Vy-unif-bound-above_zero_measure}
		\end{align}
		These two estimates follow by Assumption \ref{assumption:results:regularity_of_V} \ref{item:assumption:slow_regularity:controlled_growth} and \eqref{eq:proof-CP:Vx-unif-bound-below_first} and \eqref{eq:proof-CP:Vy-unif-bound-above_first}. 
		
		The continuity estimate of Assumption \ref{assumption:results:regularity_of_V} \ref{item:assumption:slow_regularity:continuity_estimate} yields that
		\begin{equation*}
			\liminf_{\alpha \rightarrow \infty} \Lambda(x_{\varepsilon,\alpha},p^1_{\varepsilon,\alpha},\theta_{\varepsilon,\alpha}^0) -  \Lambda(y_{\varepsilon,\alpha},p^2_{\varepsilon,\alpha},\theta_{\varepsilon,\alpha}^0) \leq 0.
		\end{equation*}
		This means that there exists a subsequence, also denoted by $\alpha$ such that
		\begin{equation} \label{eqn:Lambda_estimate_via_continuity_estimate}
			\sup_{\alpha} \Lambda(x_{\varepsilon,\alpha},p^1_{\varepsilon,\alpha},\theta_{\varepsilon,\alpha}^0) -  \Lambda(y_{\varepsilon,\alpha},p^2_{\varepsilon,\alpha},\theta_{\varepsilon,\alpha}^0) < \infty.
		\end{equation}
		Thus, we can estimate \eqref{eqn:fundamental_bootstrap_split} by \eqref{eqn:Lambda_estimate_via_continuity_estimate} and \eqref{eq:proof-CP:Vy-unif-bound-above_zero_measure}. This implies that \eqref{eqn:fundamental_controlI} holds for the chosen subsequences $\alpha$ and that for these  the collection $(x_{\varepsilon,\alpha},y_{\varepsilon,\alpha},\theta_{\varepsilon,\alpha})$ is fundamental for $\Lambda$ with respect to $\Psi$ establishing Step 3.

		\smallskip

		\underline{\emph{Proof of Step 4}}:
		
		For the subsequences constructed in Step 3, we have by \eqref{eqn:fundamental_controlI} that
		\begin{equation} \label{eqn:fundamental_controlI_sketch}
			\sup_{\alpha} \cI(x_{\varepsilon,\alpha},\theta_{\varepsilon,\alpha}) < \infty.
		\end{equation}
		As established in Step 1, following Lemma \ref{lemma:doubling_lemma}, for each $\varepsilon > 0$ the set $\{(x_{\varepsilon,\alpha},y_{\varepsilon,\alpha})\}$ is relatively compact where $\alpha$ varies over the subsequences selected in Step 3. In addition, for each $\varepsilon > 0$ there exists $z \in E$ such that $(x_{\varepsilon,\alpha},y_{\varepsilon,\alpha}) \rightarrow (z,z)$. It follows by \eqref{eqn:fundamental_controlI_sketch} and Assumption \ref{assumption:results:regularity_I} \ref{item:assumption:I:finiteness} that also
		\begin{equation} \label{eqn:fundamental_controlI_sketch2}
			\sup_{\alpha} \cI(y_{\varepsilon,\alpha},\theta_{\varepsilon,\alpha}) < \infty.
		\end{equation}
		With the bounds~\eqref{eqn:fundamental_controlI_sketch} and~\eqref{eqn:fundamental_controlI_sketch2}, the estimate~\ref{item:assumption:I:equi-cont} is satisfied for the subsequences $(x_{\varepsilon,\alpha},y_{\varepsilon,\alpha},\theta_{\varepsilon,\alpha})$. 
	\end{proof}

	\section{Existence of viscosity solutions}
	\label{section:construction-of-viscosity-solutions}

	In this section, we will prove Theorem \ref{theorem:existence_of_viscosity_solution}. In other words, we show that for~$h\in C_b(E)$ and~$\lambda>0$, the function~$R(\lambda)h$ given by
	\begin{equation*}
		R(\lambda) h(x) = \sup_{\substack{\gamma \in \mathcal{A}\mathcal{C} \\ \gamma(0) = x}} \int_0^\infty \lambda^{-1} e^{-\lambda^{-1}t} \left[h(\gamma(t)) - \int_0^t \mathcal{L}(\gamma(s),\dot{\gamma}(s))\right] \, \dd t
	\end{equation*}
	is indeed a viscosity solution to $f - \lambda \bfH f = h$. To do so, we will use the methods of Chapter 8 of \cite{FK06}.	For this strategy, one needs to check three properties of $R(\lambda)$:
	\begin{enumerate}[(a)]
		\item For all $(f,g) \in \bfH$, we have $f = R(\lambda)(f - \lambda g)$.
		\item The operator $R(\lambda)$ is a pseudo-resolvent: for all $h \in C_b(E)$ and $0 < \alpha < \beta$ we have 
		\begin{equation*}
			R(\beta)h = R(\alpha) \left(R(\beta)h - \alpha \frac{R(\beta)h - h}{\beta} \right).
		\end{equation*}
		\item The operator $R(\lambda)$ is contractive. 
	\end{enumerate}
	Thus, if $R(\lambda)$ serves as a classical left-inverse to $\bONE - \lambda \bfH$ and is also a pseudo-resolvent, then it is a viscosity right-inverse of $(\bONE- \lambda \bfH)$. For a second proof of this statement, outside of the control theory context, see Proposition 3.4 of \cite{Kr19}.
	\smallskip
	
	Establishing (c) is straightforward. The proof of (a) and (b) stems from two main properties of exponential random variable. Let $\tau_\lambda$ be the measure on $\bR^+$ corresponding to the exponential random variable with mean $\lambda^{-1}$.
	\begin{itemize}
		\item (a) is related to integration by parts: for bounded measurable functions $z$ on $\bR^+$, we have
		\begin{equation*}
			\lambda \int_0^\infty  z(t) \, \tau_\lambda( \dd t) = \int_0^\infty \int_0^t z(s) \, \dd s \, \tau_\lambda(\dd t).
		\end{equation*}
		\item (b) is related to a more involved integral property of exponential random variables. For $0 < \alpha < \beta$, we have
		\begin{multline*}
			\int_0^\infty z(s) \tau_\beta(\dd s) \\
			= \frac{\alpha}{\beta} \int_0^\infty z(s) \tau_\alpha(\dd s) + \left(1 - \frac{\alpha}{\beta}\right) \int_0^\infty \int_0^\infty z(s+u) \, \tau_\beta(\dd u) \, \tau_\alpha(\dd s).
		\end{multline*}
	\end{itemize}
	Establishing (a) and (b) can then be reduced by a careful analysis of optimizers in the definition of $R(\lambda)$, and concatenation or splittings thereof. This was carried out in Chapter 8 of \cite{FK06} on the basis of three assumptions, namely \cite[Assumptions 8.9, 8.10 and 8.11]{FK06}. We verify these below.

	\begin{proof}[Verification of Conditions 8.9, 8.10 and 8.11]
		In the notation of \cite{FK06}, we use $U = \bR^d$, $\Gamma = E \times U$, one operator $\bfH = \bfH_\dagger = \bfH_\ddagger$  and $Af(x,u) = \ip{\nabla f(x)}{u}$ for $f \in \mathcal{D}(\mathbf{H}) = C_{cc}^\infty(E)$.
		
		\smallskip
		
		Regarding Condition~8.9, by continuity and convexity of $\cH$ obtained in Propositions \ref{prop:reg-of-H-and-L:reg-H} and \ref{prop:reg-of-H-and-L:continuity}, parts 8.9.1, 8.9.2, 8.9.3 and 8.9.5 can be proven e.g. as in the proof of \cite[Lemma 10.21]{FK06} for $\psi = 1$. Part 8.9.4 is a consequence of the existence of a containment function, and follows as shown in the proof of Theorem~A.17 of \cite{CoKr17}. Since we use the argument further below, we briefly recall it here. We need to show that for any compact set $K \subseteq E$, any finite time $T > 0$ and finite bound $M \geq 0$, there exists a compact set $K' = K'(K,T,M) \subseteq E$ such that for any absolutely continuous path $\gamma :[0,T] \to E$ with $\gamma(0) \in K$, if
		\begin{equation} \label{eqn:control_on_L}
			\int_0^T \mathcal{L}(\gamma(t),\dot{\gamma}(t)) \, dt \leq M,
		\end{equation}
		then $\gamma(t) \in K'$ for any $0\leq t \leq T$.
		\smallskip
		
		For $K\subseteq E$, $T>0$, $M\geq 0$ and $\gamma$ as above, this follows by noting that
		\begin{align}
			\label{eq:action_integral_representation:Lyapunov_bound}
			\Upsilon(\gamma(\tau)) &=
			\Upsilon(\gamma(0)) + \int_0^\tau \nabla\Upsilon(\gamma(t)) \dot{\gamma}(t) \, dt \notag	\\
			&\leq \Upsilon(\gamma(0)) + \int_0^\tau \left[
			\mathcal{L}(\gamma(t),\dot{\gamma}(t))) + \mathcal{H}(x(t),\nabla \Upsilon(\gamma(t)))
			\right] \, dt \notag	\\
			&\leq \sup_K \Upsilon + M + T \sup_{x \in E} \mathcal{H}(x,\nabla \Upsilon(x)) =: C < \infty,
		\end{align}
		for any $0 \leq \tau \leq T$, so that the compact set $K' := \{z \in E \,:\, \Upsilon(z) \leq C\}$ satisfies the claim.
		
		\smallskip
		
		We proceed with the verification of Conditions~8.10 and 8.11 of~\cite{FK06}. By Proposition~\ref{prop:reg-of-H-and-L:reg-H}, we have $\cH(x,0) = 0$ and hence the application of $\bfH$ to constant $1$ function $\bONE$ satisfies $\bfH \bONE = 0$. Thus, Condition 8.10 is implied by Condition 8.11 (see Remark 8.12 (e) in~\cite{FK06}).
		
		\smallskip
		
		We establish that Condition 8.11 is satisfied: for any function $f\in \mathcal{D}(\bfH) = C_{cc}^\infty(E)$ and $x_0 \in E$, there exists an absolutely continuous path $x:[0,\infty) \to E$ such that $x(0) = x_0$ and for any $t \geq 0$,
		\begin{equation}
			\label{eq:action_integral_representation:solution_control_problem}
			\int_0^t \mathcal{H}(x(s),\nabla f(x(s)) \, ds =
			\int_0^t \left[
			\dot{x}(s) \cdot \nabla f(x(s)) - \mathcal{L}(x(s),\dot{x}(s))
			\right] \, ds.
		\end{equation}
		To do so, we solve the differential inclusion
		\begin{equation}
			\label{eq:action_integral_representation:subdifferential_eq}
			\dot{x}(t) \in \partial_p \mathcal{H}(x(t),\nabla f(x(t))), \qquad x(0) = x_0,
		\end{equation}
		where the subdifferential of $\cH$ was defined in \eqref{eqn:subdifferential} on page \pageref{eqn:subdifferential}.
		\smallskip
		
		Since the addition of a constant to $f$ does not change the gradient, we may assume without loss of generality that $f$ has compact support. A general method to establish existence of differential inclusions $\dot{x} \in F(x)$ is given by Lemma 5.1 of Deimling \cite{De92}. We have included this result as Lemma \ref{lemma:solve_differential_inclusion}, and corresponding preliminary definitions in Appendix \ref{appendix:differential_inclusions}. We use this result for $F(x) := \partial_p \cH(x,\nabla f(x))$. To apply Lemma \ref{lemma:solve_differential_inclusion}, we need to verify that:
		\begin{enumerate}[(F1)]
			\item $F$ is upper hemi-continuous and $F(x)$ is non-empty, closed, and convex for all $x \in E$.
			\item $\|F(x)\| \leq c(1 + |x|)$ on $E$, for some $c > 0$.
			\item $F(x) \cap T_E(x) \neq \emptyset$ for all $x \in E$. (For the definition of $T_E$, see Definition \ref{definition:tangent_cone} on page \pageref{eqn:subdifferential}). 
		\end{enumerate}
		Part (F1) follows from the properties of subdifferential sets
		of convex and continuous functionals. $\cH$ is continuous in $(x,p)$ and convex in $p$ by Proposition \ref{prop:reg-of-H-and-L:reg-H}.
		Part (F3) is a consequence of Lemma \ref{lemma:Hamiltonian_vector_flow_from_Lambda}, which yields that~$F(x)\subseteq T_E(x)$. 
		Part (F2) is in general not satisfied. To circumvent this problem, we use properties of $\cH$ to establish a-priori bounds on the range of solutions.
		
		\smallskip
		
		\emph{Step 1:} Let $T > 0$, and assume that $x(t)$ solves \eqref{eq:action_integral_representation:subdifferential_eq}. We establish that there is some $M$ such that~\eqref{eqn:control_on_L} is satisfied. By~\eqref{eq:action_integral_representation:subdifferential_eq} we obtain for all $p \in \mathbb{R}^d$,
		\[
		\mathcal{H}(x(t),p) \geq \mathcal{H}(x(t),\nabla f(x(t))) + \dot{x}(t) \cdot (p - \nabla f(x(t))),
		\]
		and as a consequence
		\[
		\dot{x}(t) \nabla f(x(t)) - \mathcal{H}(x(t),\nabla f(x(t))) \geq
		\mathcal{L}(x(t),\dot{x}(t)).
		\]
		Since $f$ has compact support and $\mathcal{H}(y,0) = 0$ for any $y \in E$, we estimate 
		\begin{align*}
			\int_0^T \mathcal{L}(x(t),\dot{x}(t)) \, ds &\leq
			\int_0^T \dot{x}(t) \nabla f(x(t)) \, dt - T\inf_{y \in \mathrm{supp}(f)} \mathcal{H}(y,\nabla f(y)).
		\end{align*}
		By continuity of $\mathcal{H}$ the field $F$ is bounded on compact sets, so the first term can be bounded by
		\[
		\int_0^T \dot{x}(t) \nabla f(x(t)) \, dt \leq
		T \sup_{y \in \mathrm{supp}(f)}\|F(y)\| \sup_{z \in \mathrm{supp}(f)}|\nabla f(z)|.
		\]
		Therefore, for any $T>0$, we obtain that the integral over the Lagrangian is bounded from above by $M = M(T)$, with
		\[
		M := T \sup_{y \in \mathrm{supp}(f)}\|F(y)\| \sup_{z \in \mathrm{supp}(f)}|\nabla f(z)| -
		\inf_{y \in \mathrm{supp}(f)} \mathcal{H}(y,\nabla f(y)).
		\]
		From the first part of the, see the argument concluding after \eqref{eq:action_integral_representation:Lyapunov_bound}, we find that the solution $x(t)$ remains in the compact set
		\begin{equation} \label{eqn:containment_set_existence}
			K' := \left\{
			z \in E \, \middle| \, \Upsilon(z) \leq C 
			\right\}, \quad C := \Upsilon(x_0) + M + T \sup_x \cH(x,\nabla \Upsilon(x)),
		\end{equation}
		for all $t \in [0,T]$.
		
		\smallskip
		
		\emph{Step 2}: We prove that there exists a solution $x(t)$ of  \eqref{eq:action_integral_representation:subdifferential_eq} on $[0,T]$. 

		Using $F$, we define a new multi-valued vector-field $F'(z)$ that equals $F(z) = \partial_p \mathcal{H}(z,\nabla f(z))$ inside $K'$, but equals $\{0\}$ outside a neighborhood of $K$. This can e.g. be achieved by multiplying with a smooth cut-off function $g_{K'} : E \to [0,1]$ that is equal to one on $K'$ and zero outside of a neighborhood of $K'$.
		\smallskip
		
		The field $F'$ satisfies (F1), (F2) and (F3) from above, and hence there exists an absolutely continuous path $y : [0,\infty) \to E$ such that $y(0) = x_0$ and for almost every $t \geq 0$,
		\[
		\dot{y}(t) \in F'(y(t)).
		\]
		By the estimate established in step 1 and the fact that $\Upsilon(\gamma(t)) \leq C$ for any $0 \leq t \leq T$, it follows from the argument as shown above in \eqref{eq:action_integral_representation:Lyapunov_bound} that the solution $y$ stays in $K'$ up to time $T$. Since on $K'$, we have $F' = F$, this implies that setting $x = y|_{[0,T]}$, we obtain a solution $x(t)$ of \eqref{eq:action_integral_representation:subdifferential_eq} on the time interval $[0,T]$.
	\end{proof}
	
	\begin{lemma} \label{lemma:Hamiltonian_vector_flow_from_Lambda}
		Let Assumption~\ref{assumption:Hamiltonian_vector_field} be satisfied. Then the map $\cH : E \times \bR^d \rightarrow \bR$ defined in~\eqref{eq:results:variational_hamiltonian} is such that $\partial_p \cH(x,p) \subseteq T_E(x)$ for all $p$ and $x \in E$.
	\end{lemma}

	\begin{proof}
		Fix $x \in E$ and $p_0 \in \bR^d$. We aim to prove that $\partial_p \cH(x,p_0) \subseteq T_E(x)$. Recall the definition of $\cH$:
		\begin{equation} \label{eqn:Hamiltonian_in_subdifferential_proof}
			\cH(x,p) = \sup_{\theta \in \Theta} \left\{\Lambda(x,p,\theta) - \cI(x,\theta) \right\}.
		\end{equation}
		Let $\Omega(p) \subseteq \Theta$ be the set of controls that optimize $\cH$: thus if $\theta \in \Omega(p)$ then $\cH(x,p) = \Lambda(x,p,\theta) - \cI(x,\theta)$.
		\smallskip 
		
		The result will follow from the following claim,
		\begin{equation} \label{eqn:proof_hamflow_inclusion}
			\partial_p \cH(x,p_0) = ch \bigcup_{\theta \in \Omega(p_0)} \partial_p \Lambda(x,p_0,\theta),
		\end{equation}
		where $ch$ denotes the convex hull. Having established this claim, the result follows from Assumption \ref{assumption:Hamiltonian_vector_field} and the fact that $T_E(x)$ is a convex set by Lemma \ref{proposition:tangent_cone_convex}.
		
		\smallskip

		We start with the proof of \eqref{eqn:proof_hamflow_inclusion}. For this we will use~\cite[Theorem D.4.4.2]{HULe01}. To study the subdifferential of the function $\partial_{p} \cH(x,p_0)$, it suffices to restrict the domain of the map $p \mapsto \cH(x,p)$ to the closed ball $B_1(p_0)$ around $p_0$ with radius $1$.
		
		To apply \cite[Theorem D.4.4.2]{HULe01} for this restricted map, first recall that $\Lambda$ is continuous by Assumption \ref{assumption:results:regularity_of_V} \ref{item:assumption:slow_regularity:continuity} and that $\cI$ is lower semi-continuous by Assumption \ref{assumption:results:regularity_I} \ref{item:assumption:I:lsc}. Secondly, we need to find a compact set $\Omega \subseteq \Theta$ such that we can restrict the supremum (for any $p \in B_1(p_0))$ in \eqref{eqn:Hamiltonian_in_subdifferential_proof} to $\Omega$:
		\begin{equation*}
			\cH(x,p) = \sup_{\theta \in \Omega} \left\{\Lambda(x,p,\theta) - \cI(x,\theta) \right\}.
		\end{equation*}
		In particular, we show that 
		we can take for $\Omega$ a sublevelset of $\cI(x,\cdot)$ which is compact by Assumption \ref{assumption:results:regularity_I} \ref{item:assumption:I:compact-sublevelsets}. 
		
		\smallskip
		
		Let $\theta_{x}^0$ be the control such that $\cI(x,\theta_{x}^0) = 0$, which exists due to Assumption \ref{assumption:results:regularity_I} \ref{item:assumption:I:zero-measure}. Let $M^*$ be such that (with the constants $M,C_1,C_2$ as in Assumption \ref{assumption:results:regularity_of_V} \ref{item:assumption:slow_regularity:controlled_growth})
		\begin{equation*}
			M^* = \sup_{p \in B_1(p_0)} \left\{ \max\left\{M,C_1 \Lambda(x,p,\theta_{x}^0) + C_2\right\} - \Lambda(x,p,\theta_{x}^0) \right\} < \infty.
		\end{equation*}
		Note that $M^*$ is finite as $p \mapsto \Lambda(x,p,\theta_{x}^0)$ is continuous on the closed unit ball $B_1(p_0)$. Then we find, due to Assumption \ref{assumption:results:regularity_of_V} \ref{item:assumption:slow_regularity:controlled_growth}, that if $\theta$ satisfies $\cI(x,\theta) > M^*$, then for any~$p\in B_1(p_0)$ we have
		\begin{equation*}
			\Lambda(x,p,\theta) - \cI(x,\theta) < \Lambda(x,p,\theta_{x}^0) \leq \cH(x_0,p).
		\end{equation*}
		We obtain that if $p \in B_1(p_0)$, then we can restrict our supremum in \eqref{eqn:Hamiltonian_in_subdifferential_proof} to the compact set $\Omega := \Theta_{\{x\},M^*}$, see Assumption \ref{assumption:results:regularity_I} \ref{item:assumption:I:compact-sublevelsets}.
		
		\smallskip
		Thus, it follows by \cite[Theorem D.4.4.2]{HULe01} that 
		\begin{equation*}
			\partial_p \cH(x,p_0) = ch \left(\bigcup_{\theta \in \Theta_{\{x\},\bar{M}^*}} \partial_p \left(\Lambda(x,p_0,\theta) - \cI(x,\theta) \right)\right),
		\end{equation*}
		where $ch$ denotes the convex hull. Now \eqref{eqn:proof_hamflow_inclusion} follows by noting that $\cI(x,\theta)$ does not depend on $p$. 
	\end{proof}

	\section{Examples of Hamiltonians}
	\label{section:verification-for-examples-of-Hamiltonians}
	In this section we specify our general results to two concrete examples of Hamiltonians of the type 
	\begin{equation}\label{eq:results:H-example-section}
		\mathcal{H}(x,p) = \sup_{\theta \in \Theta}\left[\Lambda(x,p,\theta) - \mathcal{I}(x,\theta)\right].
	\end{equation}
	The purpose of this section is to showcase that the method introduced in this paper is versatile enough to capture interesting examples that could not be treated before.
	
	\smallskip

	First, we consider in Proposition~\ref{prop:diffusion-coupled-to-jumps} Hamiltonians that one encounters in the large deviation analysis of two-scale systems as studied in \cite{BuDuGa18} and \cite{KuPo17} when considering a diffusion process coupled to a fast jump process. Second,  we consider in Proposition~\ref{prop:mean-field-coupled-to-diffusion} the example treated in our introduction that arises from models of mean-field interacting particles that are coupled to fast external variables. This example will be further analyzed in \cite{KrSchl20}.
	\begin{proposition}[Diffusion coupled to jumps]\label{prop:diffusion-coupled-to-jumps}
		Let $E=\mathbb{R}^d$ and $F=\{1,\dots,J\}$ be a finite set. Suppose the following:
		\begin{enumerate}[label=(\roman*)]
			\item The set of control variables is $\Theta:=\mathcal{P}(\{1,\dots,J\})$, that is probability measures over the finite set $F$. 
			\item The function $\Lambda$ is given by
			\begin{equation*} 
				\Lambda(x,p,\theta) := \sum_{i\in F}\left[\ip{a(x,i)p}{p}+\ip{b(x,i)}{p}\right]\theta_i,
			\end{equation*}
			where $a:E\times F\to\mathbb{R}^{d\times d}$ and $b:E\times F\to\mathbb{R}^d$, and $\theta_i:=\theta(\{i\})$.
			\item The cost function $\mathcal{I}:E\times\Theta\to[0,\infty)$ is given by
			\begin{equation*}
				\mathcal{I}(x,\theta) := \sup_{w\in\mathbb{R}^J}\sum_{ij}r(i,j,x)\theta_i \left[1-e^{w_j-w_i}\right],
			\end{equation*}
			with non-negative rates $r:F^2\times E\to[0,\infty)$.
		\end{enumerate}
		Suppose that the cost function~$\mathcal{I}$ satisfies the assumptions of Proposition~\ref{prop:verify:DV-for-Jumps} below and the function~$\Lambda$ satisfies the assumptions of Proposition~\ref{prop:verify-ex:Lambda_quadratic} below. Then Theorems~\ref{theorem:comparison_principle_variational} and~\ref{theorem:existence_of_viscosity_solution} apply to the Hamiltonian~\eqref{eq:results:H-example-section}.
	\end{proposition}
	
	\begin{proof}
		To apply Theorems \ref{theorem:comparison_principle_variational} and \ref{theorem:existence_of_viscosity_solution}, we need to verify Assumptions \ref{assumption:results:regularity_of_V}, \ref{assumption:results:regularity_I} and \ref{assumption:Hamiltonian_vector_field}.
		Assumption \ref{assumption:results:regularity_of_V} follows from Proposition~\ref{prop:verify-ex:Lambda_quadratic}, Assumption \ref{assumption:results:regularity_I} follows from Proposition~\ref{prop:verify:DV-for-Jumps} and Assumption \ref{assumption:Hamiltonian_vector_field} is satisfied as $E = \bR^d$.
	\end{proof}
	
	\begin{remark}
		We assume uniform ellipticity of $a$, which we use to establish \ref{item:assumption:slow_regularity:controlled_growth}. This leaves our comparison principle slightly lacking to prove a large deviation principle as general as in \cite{BuDu19}. In contrast, we do not need a Lipschitz condition on $r$ in terms of $x$.
		
		While we believe that the conditions on $a$ can be relaxed by performing a finer analysis of the estimates in terms of $a$, we do not pursue this relaxation here. 
	\end{remark}
	
	\begin{remark}
		The cost function is the large deviation rate function for the occupation time measures of a jump process taking values in a finite set $\{1,\dots,J\}$, see e.g.~\cite{DoVa75b,Ho08}.
	\end{remark}
	
	\begin{remark}
		In the context with $a = 0$ and $\cI$ as general as Assumption \ref{assumption:results:regularity_I}, we improve upon the results of Chapter III of \cite{BaCD97} by allowing a more general class of functionals $\cI$, that are e.g. discontinuous as for example in Proposition \ref{prop:mean-field-coupled-to-diffusion} below.
		
		\smallskip
		
		In \cite{DLLe11} the authors consider a second order Hamilton-Jacobi-Bellman equation, with the quadratic part replaced by a second order part. They work, however, with continuous cost functional $\cI$. An extension of \cite{DLLe11} that allows for a similar flexibility in the choice of $\cI$ would therefore be of interest.
	\end{remark}
	
	\begin{remark}
		Under irreducibility conditions on the rates, as we shall assume below in Proposition~\ref{prop:verify:DV-for-Jumps}, by~\cite{DoVa75a} the Hamiltonian~$\mathcal{H}(x,p)$ is the principal eigenvalue of the matrix $A_{x,p} \in \mathrm{Mat}_{J \times J}(\mathbb{R})$ given by
		\[
		A_{x,p} = \mathrm{diag}\left[\ip{a(x,1)p}{p}+\ip{b(x,1)}{p}, \dots, \ip{a(x,J)p}{p}+\ip{b(x,J)}{p}\right] + R_x,
		\]
		where $x,p \in \mathbb{R}^d$ and $R_x$ is the matrix
		\[
		\begin{pmatrix}
			-\sum_{j \neq 1} r(1,j,x)	&	r(1,2,x)	&	\dots	&	r(1,J,x)\\
			r(2,1,x)	&	-\sum_{j \neq 2} r(2,j,x)	&	\dots	&	\vdots\\
			\vdots	&	\vdots	&	\ddots	&	\vdots \\
			r(J,1,x)	&	\dots	&	r(J,J-1,x)	&	-\sum_{j \neq J}r(J,j,x)
		\end{pmatrix},
		\]
		that is $(R_x)_{ii} = -\sum_{j \neq i} r(i,j,x)$ on the diagonal and $(R_x)_{ij} = r(i,j,x)$ for $i \neq j$.
	\end{remark}
	
	Next we consider Hamiltonians arising in the context of weakly interacting jump processes on a collection of states $\{1,\dots,q\}$ as described in our introduction. We analyze and motivate this example in more detail in our companion paper \cite{KrSchl20}. We give the terminology as needed for the results in this paper.
	
	\smallskip
	
	The empirical measure of the interacting processes takes its values in the set of measures $\cP(\{1,\dots,q\})$. The dynamics arises from mass moving over the bonds $(a,b) \in \Gamma = \left\{(i,j) \in \{1,\dots,q\}^2 \, | \, i \neq j\right\}$. As the number of processes is send to infinity, there is a type of limiting result for the total mass moving over the bonds. 
	
	\smallskip 
	
	We will denote by $v(a,b,\mu,\theta)$ the total mass that moves from $a$ to $b$ if the empirical measure equals $\mu$ and the control is given by $\theta$. We will make the following assumption on the kernel $v$.
	
	\begin{definition}[Proper kernel] \label{definition:proper_kernel}
		Let $v : \Gamma \times \cP(\{1,\dots,q\}) \times \Theta \rightarrow \bR^+$. We say that $v$ is a \textit{proper kernel} if $v$ is continuous and if for each $(a,b) \in \Gamma$, the map $(\mu,\theta) \mapsto v(a,b,\mu,\theta)$ is either identically equal to zero or satisfies the following two properties:
		\begin{enumerate}[(a)]
			\item $v(a,b,\mu,\theta) = 0$ if $\mu(a) = 0$ and $v(a,b,\mu,\theta) > 0$ for all $\mu$ such that $\mu(a) > 0$. 
			\item There exists a decomposition $v(a,b,\mu,\theta) = v_{\dagger}(a,b,\mu(a)) v_{\ddagger}(a,b,\mu,\theta)$ such that $v_{\dagger}$ is increasing in the third coordinate and such that $v_{\ddagger}(a,b,\cdot,\cdot)$ is continuous and satisfies $v_{\ddagger}(a,b,\mu,\theta) > 0$.
		\end{enumerate}
	\end{definition}
	A typical example of a proper kernel is given by 
	\begin{equation*}
		v(a,b,\mu,\theta) = \mu(a) r(a,b,\theta) e^{ \partial_a V(\mu) - \partial_b V(\mu)}, 
	\end{equation*}
	with $r > 0$ continuous and $V \in C^1_b(\cP(\{1,\dots,q\}))$.
	\begin{proposition}[Mean-field coupled to diffusion]\label{prop:mean-field-coupled-to-diffusion}
		Let the space $E$ be given by the embedding of $E:=\mathcal{P}(\{1,\dots,J\})\times[0,\infty)^\Gamma\subseteq \mathbb{R}^d$ and $F$ be a smooth compact Riemannian manifold without boundary. Suppose the following. 
		\begin{enumerate}[label=(\roman*)]
			\item The set of control variables $\Theta$ equals $\mathcal{P}(F)$.
			\item The function $\Lambda$ is given by
			\begin{equation*} 
				\Lambda((\mu,w),p,\theta) = \sum_{(a,b) \in \Gamma} v(a,b,\mu,\theta)\left[\exp\left\{p_b - p _a + p_{(a,b)} \right\} - 1 \right]
			\end{equation*}
			with a proper kernel $v$ in the sense of Definition~\ref{definition:proper_kernel}.
			\item The cost function $\mathcal{I}:E\times\Theta\to[0,\infty]$ is given by
			\begin{equation*}
				\mathcal{I}(x,\theta) := \sup_{\substack{u\in \mathcal{D}(L_x)\\ \inf u > 0}}\left[ -\int_F \frac{L_x u}{u}\,d\theta\right],
			\end{equation*}
			where $L_x$ is a second-order elliptic operator locally of the form
			\begin{equation*}
				L_x = \frac{1}{2}\nabla\cdot\left(a_x \nabla\right) + b_x\cdot \nabla,
			\end{equation*}
			on the domain $\mathcal{D}(L_x):=C^2(F)$, with positive-definite matrix $a_x$ and co-vectors $b_x$.
		\end{enumerate}
		Suppose that the cost function~$\mathcal{I}$ satisfies the assumptions of Proposition~\ref{prop:verify:DV-functional-of-drift-diffusion} and the function~$\Lambda$ satisfies the assumptions of Proposition~\ref{prop:verify-ex:Lambda_exponential}. Then Theorems~\ref{theorem:comparison_principle_variational} and~\ref{theorem:existence_of_viscosity_solution} apply to the Hamiltonian~\eqref{eq:results:H-example-section}. 
	\end{proposition}
	\begin{proof}
		To apply Theorems \ref{theorem:comparison_principle_variational} and \ref{theorem:existence_of_viscosity_solution}, we need to verify Assumptions \ref{assumption:results:regularity_of_V}, \ref{assumption:results:regularity_I} and \ref{assumption:Hamiltonian_vector_field}.
		Assumption \ref{assumption:results:regularity_of_V} follows from Proposition~\ref{prop:verify-ex:Lambda_exponential} and Assumption \ref{assumption:results:regularity_I} follows from Proposition~\ref{prop:verify:DV-functional-of-drift-diffusion}. We verify Assumption \ref{assumption:Hamiltonian_vector_field} in Proposition \ref{proposition:Ham_flow_for_exp_hamiltonian}.
	\end{proof}
	\begin{remark}
		The cost function stems from occupation-time large deviations of a drift-diffusion process on a compact manifold, see e.g. \cite{DoVa75a,Pi07}. We expect Proposition~\ref{prop:mean-field-coupled-to-diffusion} to extend also to non-compact spaces $F$, but we feel this technical extension is better suited for a separate paper.
	\end{remark}

	\subsection{Verifying assumptions for cost functions \texorpdfstring{$\mathcal{I}$}{I}}
	\label{section:verify-ex:cost-functions}

	Here we verify Assumption~\ref{assumption:results:regularity_I} for two types of cost functions $\mathcal{I}(x,\theta)$ appearing in the  examples of  Propositions~\ref{prop:diffusion-coupled-to-jumps} and~\ref{prop:mean-field-coupled-to-diffusion}.
	
	\begin{proposition}[Donsker-Varadhan functional for jump processes]
		\label{prop:verify:DV-for-Jumps}
		Consider a finite set $F = \{1,\dots,J\}$ and let $\Theta := \mathcal{P}(\{1,\dots,J\})$ be the set of probability measures on $F$. For $x\in E$, let $L_x : C_b(F) \rightarrow C_b(F)$ be the operator given by
		\begin{equation*}
			L_x f(i) := \sum_{j=1}^Jr(i,j,x)\left[f(j)-f(i)\right],\quad f :\{1,\dots,J\}\to\mathbb{R}.
		\end{equation*}
		Suppose that the rates $r:\{1,\dots,J\}^2\times E \rightarrow \bR^+$ are continuous as a function on $E$ and moreover satisfy the following:
		\begin{enumerate}[label=(\roman*)]
			\item For any $x\in E$, the matrix $R(x)$ with entries $R(x)_{ij} := r(i,j,x)$ for $i\neq j$ and $R(x)_{ii} = -\sum_{j\neq i}r(i,j,x)$ is irreducible.
			\item For each pair $(i,j)$, we either have $r(i,j,\cdot)\equiv 0$ or for each compact set $K\subseteq E$, it holds that
			\begin{equation*}
				r_{K}(i,j) := \inf_{x\in K}r(i,j,x) > 0.
			\end{equation*}
		\end{enumerate}
		Then the Donsker-Varadhan functional $\mathcal{I}:E\times\Theta \rightarrow \bR^+$ defined by
		\begin{align*}
			\mathcal{I}(x,\theta) & :=  -\inf_{\substack{\phi \in C_b(F) \\ \inf \phi > 0}} \int \frac{L_x \phi(z)}{\phi(z)} \, \dd \theta \\
			& = \sup_{w\in\mathbb{R}^J}\sum_{ij}r(i,j,x)\theta_i \left[1-e^{w_j-w_i}\right]
		\end{align*}
		satisfies Assumption~\ref{assumption:results:regularity_I}.
	\end{proposition}

	\begin{proof}[Proof]
		\underline{\ref{item:assumption:I:lsc}}: For a fixed vector $w\in\mathbb{R}^J$, the map
		\begin{equation*}
			(x,\theta)\mapsto \sum_{ij}r(i,j,x)\theta_i \left[1-e^{w_j-w_i}\right]
		\end{equation*}
		is continuous on $E\times\Theta$. Hence $\mathcal{I}(x,\theta)$ is lower semicontinuous as the supremum over continuous functions.
		\smallskip
		
		\underline{\ref{item:assumption:I:zero-measure}}: Let $x\in E$. First note that for all $\theta$, the choice $w = 0$ implies that $\cI(x,\theta) \geq 0$. By the irreducibility assumption on the rates $r(i,j,x)$, there exists a unique measure $\theta_{x}^0\in\Theta$ such that for any $f:\{1,\dots,J\}\to\mathbb{R}$,
		\begin{equation} \label{eqn:example_jump_DV_stationarity}
			\sum_i L_x f(i) \theta_{x}^0(i)=0.
		\end{equation}
		We establish that $\cI(x,\theta_{x}^0) = 0$. Let $w \in \mathbb{R}^J$. By the elementary estimate
		\begin{equation*}
			\left(1-e^{b - a}\right)\leq -(b-a) \quad \text{ for all } \; a,b > 0,
		\end{equation*}
		we obtain that
		\begin{align*}
			\sum_{ij}r(i,j,x) \theta_{x}^0(i) \left(1-e^{w_j - w_i}\right)
			&\leq  \sum_{ij}r(i,j,x) \theta_{x}^0(i) \left(w_j - w_i \right)\\
			&= \sum_i (L_x w)(i) \theta_{x}^0(i) = 0
		\end{align*}
		by \eqref{eqn:example_jump_DV_stationarity}. Since $\mathcal{I} \geq 0$, this implies $\mathcal{I}(x,\theta_{x}^0) = 0$.
		\smallskip
		
		\underline{\ref{item:assumption:I:compact-sublevelsets}}: Any closed subset of $\Theta$ is compact.
		\smallskip
		
		\underline{\ref{item:assumption:I:finiteness}}: Let $x_n\to x$ in $E$. It follows that the sequence is contained in some compact set $K \subseteq E$ that contains the $x_n$ and $x$ in its interior. For any $y\in K$,
		\begin{equation*}
			\mathcal{I}(y,\theta) \leq \sum_{ij, i \neq j} r(i,j,y) \theta_i \leq
			\sum_{ij, i\neq j} r(i,j,y) \leq
			\sum_{ij, i \neq j} \bar{r}_{ij}, \quad \bar{r}_{ij} := \sup_{y \in K} r(i,j,y).
		\end{equation*}
		Hence $\mathcal{I}$ is uniformly bounded on $K\times\Theta$, and~\ref{item:assumption:I:finiteness} follows with $U_x$ the interior of $K$. 
		\smallskip
		
		\underline{\ref{item:assumption:I:equi-cont}}: Let $d$ be some metric that metrizes the topology of $E$. We will prove that for any compact set $K\subseteq E$ and $\varepsilon > 0$ there is some $\delta > 0$ such that for all $x,y \in K$ with $d(x,y) \leq \delta$ and for all $\theta \in \mathcal{P}(F)$, we have
		\begin{equation} \label{eqn:proof_equi_cont_I}
			|\mathcal{I}(x,\theta) - \mathcal{I}(y,\theta)| \leq \varepsilon.
		\end{equation}
		Let $x,y \in K$. By continuity of the rates the $\mathcal{I}(x,\cdot)$ are uniformly bounded for $x \in K$:
		\begin{equation*}
			0 \leq \mathcal{I}(x,\theta) \leq \sum_{ij, i \neq j} r(i,j,x) \theta_i \leq
			\sum_{ij, i\neq j} r(i,j,x) \leq
			\sum_{ij, i \neq j} \bar{r}_{ij}, \quad \bar{r}_{ij} := \sup_{x \in K} r(i,j,x).
		\end{equation*}
		For any $n \in \mathbb{N}$, there exists $w^n \in \mathbb{R}^J$ such that
		\begin{equation*}
			0 \leq \mathcal{I}(x,\theta) \leq \sum_{ij, i \neq j} r_{ij}(x) \theta_i (1 - e^{w^n_j - w^n_i}) + \frac{1}{n}.
		\end{equation*}
		By reorganizing, we find for all bonds $(a,b)$ the bound
		\begin{equation*}
			\theta_a e^{w^n_b - w^n_a} \leq
			\frac{1}{r_{K,a,b}} \left[
			\sum_{ij, i \neq j} r(i,j,x)\theta_i + \frac{1}{n}
			\right]
			\leq \frac{1}{r_{K,a,b}} \left[
			\sum_{ij, i \neq j} \bar{r}_{ij} + \frac{1}{n}
			\right],
		\end{equation*}
		where~$r_{K,a,b}:=\inf_{x\in K}r(a,b,x)$.
		Thereby, evaluating in $\mathcal{I}(y,\theta)$ the same vector $w^n$ to estimate the supremum,
		\begin{align*}
			& \mathcal{I}(x,\theta) - \mathcal{I}(y,\theta) \\
			&\leq \frac{1}{n} + \sum_{ab, a\neq b} r(a,b,x) \theta_a (1 - e^{w^n_b - w^n_a}) - \sum_{ab, a\neq b} r(a,b,y) \theta_a (1 - e^{w^n_b - w^n_a})	\\
			&\leq \frac{1}{n} + \sum_{ab, a\neq b} |r(a,b,x) - r(a,b,y)| \theta_a + \sum_{ab, a\neq b} |r(a,b,y) - r(a,b,x)| \theta_a e^{w^n_b - w^n_a}	\\
			&\leq \frac{1}{n} + \sum_{ab, a\neq b}|r(a,b,x) - r(a,b,y)| \left(
			1 + \frac{1}{r_{K,a,b}} \left[\sum_{ij, i \neq j} \bar{r}_{ij} + 1 \right]
			\right).
		\end{align*}
		We take $n \to \infty$ and use that the rates $x \mapsto r(a,b,x)$ are continuous, and hence uniformly continuous on compact sets, to obtain \eqref{eqn:proof_equi_cont_I}.
	\end{proof}
	\begin{proposition}[Donsker-Varadhan functional for drift-diffusions]
		\label{prop:verify:DV-functional-of-drift-diffusion}
		Let $F$ be a smooth compact Riemannian manifold without boundary and set $\Theta:=\mathcal{P}(F)$, the set of probability measures on $F$. For $x\in E$, let $L_x : C^2(F) \subseteq C_b(F) \rightarrow C_b(F)$ be the second-order elliptic operator that in local coordinates is given by
		\begin{equation*}
			L_x = \frac{1}{2}\nabla\cdot\left(a_x \nabla\right) + b_x\cdot \nabla,
		\end{equation*}
		where $a_x$ is a positive definite matrix and $b_x$ is a vector field having smooth entries $a_x^{ij}$ and $b_x^i$ on $F$. Suppose that for all $i,j$ the maps 
		\begin{equation} \label{eqn:examples_DV_diffusion_continuous_coefficients}
			x \mapsto a_x^{i,j}(\cdot), \qquad x \mapsto b_x^i(\cdot)
		\end{equation}	
		are continuous as functions from $E$ to $C_b(F)$, where we equip $C_b(F)$ with the supremum norm. Then the functional $\mathcal{I}:E\times\Theta\to[0,\infty]$ defined by
		\begin{equation*}
			\mathcal{I}(x,\theta) := \sup_{\substack{u\in \mathcal{D}(L_x)\\u>0}}\left[ -\int_F \frac{L_xu}{u}\,d\theta\right]
		\end{equation*}
		satisfies Assumption~\ref{assumption:results:regularity_I}.
	\end{proposition}
	\begin{proof}[Proof]
		\underline{\ref{item:assumption:I:lsc}}: For any fixed function $u\in\mathcal{D}(L_x)$ such that $u > 0$, the function $(-L_xu/u)$ is continuous on $F$. Note that by definition of $\cI$ it suffices to only consider $u > 0$. Thus, for any such fixed $u > 0$ it follows by~\eqref{eqn:examples_DV_diffusion_continuous_coefficients} and compactness of $F$ that 
		\begin{equation*}
			(x,\theta)\mapsto -\int_F \frac{L_xu}{u}\,d\theta
		\end{equation*}
		is continuous on $E\times\Theta$. As a consequence $\mathcal{I}(x,\theta)$ is lower semicontinuous as the supremum over continuous functions.
		\smallskip
		
		\underline{\ref{item:assumption:I:zero-measure}}: Let $x\in E$. The stationary measure $\theta_{x}^0 \in\Theta$ satisfying
		\begin{equation}\label{eq:proof:verify-I-diffusion:I2}
			\int_F L_xg(z)\,\theta_{x}^0(\dd z) = 0\quad \text{for all}\;g\in \mathcal{D}(L_x)
		\end{equation}
		is the minimizer of $\mathcal{I}(x,\cdot)$, that is $\mathcal{I}(x,\theta_{x}^0) = 0$. This follows by considering the Hille-Yosida approximation $L_x^\varepsilon$ of $L_x$ and using the same argument (using $w = \log u$) as in Proposition~\ref{prop:verify:DV-for-Jumps} for these approximations. For any $u>0$ and $\varepsilon>0$, 
		\begin{align*}
			-\int_F \frac{L_xu}{u}\,d\theta & = -\int_F \frac{L^\varepsilon_xu}{u}\,d\theta + \int_F \frac{(L^\varepsilon_x-L_x)u}{u}\,d\theta\\
			&\leq -\int_F \frac{L^\varepsilon_xu}{u}\,d\theta + \frac{1}{\inf_F u} \|(L_x^\varepsilon-L_x)u\|_F\\
			&\leq -\int_F L^\varepsilon_x \log(u)\,d\theta + o(1).
		\end{align*}
		Sending $\varepsilon\to 0$ and then using~\eqref{eq:proof:verify-I-diffusion:I2} gives~\ref{item:assumption:I:zero-measure}.
		\smallskip
		
		\underline{\ref{item:assumption:I:compact-sublevelsets}}: Since $\Theta = \mathcal{P}(F)$ is compact, any closed subset of $\Theta$ is compact. Hence any union of sub-level sets of $\mathcal{I}(x,\cdot)$ is relatively compact in~$\Theta$.
		\smallskip

		\underline{\ref{item:assumption:I:finiteness}}:
		Fix $x \in E$ and $M \geq 0$. Let $\theta \in \Theta_{\{x\},M}$. As $\cI(x,\theta) \leq M$, we find by \cite{Pi85a} that the density $\frac{\dd \theta}{\dd z}$ exists, where $\dd z$ denotes the Riemannian volume measure. 
		
		In addition it follows from \cite[Theorem 1.4]{Pi85a} there exists constants $c_1(y),c_2(y),c_3(y),c_4(y)$, $c_1(y),c_2(y)$ being positive, depending continuously on $a_y, a_y^{-1},b_y$ (See the derivation of~\cite[Eq.~(2.18),~(2.19)]{Pi85b}), but not on $\theta$, such that
		\begin{equation} \label{eqn:Pinksy_bootstrap_init}
			c_1(y)\int_F|\nabla g_\theta |^2\,dz - c_2(y) \leq \mathcal{I}(y,\theta) \leq c_3(y) \int_F|\nabla g_\theta |^2\,dz + c_4(y),
		\end{equation}
		where $g_\theta = (\dd \theta/\dd z)^{1/2}$.
		
		As the dependence is continuous in $y$, we can find a open set $U \subseteq E$ of $x$ such that there are constants $c_1,c_2,c_3,c_4$, $c_1,c_3$ being positive, that do not depend on $\theta$, such that for any $y \in U$:
		\begin{equation} \label{eqn:Pinksy_bootstrap}
			c_1\int_F|\nabla g_\theta |^2\,dz - c_2 \leq \mathcal{I}(y,\theta) \leq c_3 \int_F|\nabla g_\theta |^2\,dz + c_4.
		\end{equation}
		From \eqref{eqn:Pinksy_bootstrap}, \ref{item:assumption:I:finiteness} immediately follows.
		
		\smallskip
		
		\underline{\ref{item:assumption:I:equi-cont}}: Since the coefficients $a_x$ and $b_x$ of the operator $L_x$ depend continuously on $x$, assumption~\ref{item:assumption:I:equi-cont} follows from Theorem~2 of~\cite{Pi07}.
	\end{proof}

	\subsection{Verifying assumptions for functions \texorpdfstring{$\Lambda$}{Lambda}}
	\label{section:verify-ex:functions-Lambda}
	
	In this section we verify Assumption~\ref{assumption:results:regularity_of_V} for two types of functions $\Lambda(x,p,\theta)$ appearing in the examples of   Propositions~\ref{prop:diffusion-coupled-to-jumps} and~\ref{prop:mean-field-coupled-to-diffusion}.
	\begin{proposition}[Quadratic function $\Lambda$]\label{prop:verify-ex:Lambda_quadratic}
		Let $E=\mathbb{R}^d$ and $\Theta=\mathcal{P}(F)$ for some compact Polish space $F$. Suppose that the function $\Lambda :E\times\mathbb{R}^d\times\Theta\to\mathbb{R}$ is given by
		\begin{equation*} 
			\Lambda(x,p,\theta) = \int_F\ip{a(x,z)p}{p}\,d\theta(z) + \int_F\ip{b(x,z)}{p}\,d\theta(z),
		\end{equation*}
		where $a:E\times F\to\mathbb{R}^{d\times d}$ and $b:E\times F\to\mathbb{R}^d$ are continuous. Suppose that for every compact set $K \subseteq \bR^d$,
		\begin{align*}
			a_{K,min} & := \inf_{x \in K, z \in F, |p|=1} \ip{a(x,z)p}{p} > 0, \\
			a_{K,max} & := \sup_{x \in K, z \in F, |p| = 1} \ip{a(x,z)p}{p} < \infty, \\
			b_{K,max} & := \sup_{x \in K, z \in F, |p|=1} |\ip{b(x,z)}{p}| < \infty.
		\end{align*}
		Furthermore, there exists a constant $L>0$ such that for all $x,y\in E$ and $z\in F$,
		\begin{equation*}
			\|a(x,z)-a(y,z)\| \leq L|x-y|,
		\end{equation*}
		and suppose that the functions $b$ are one-sided Lipschitz. Then Assumption~\ref{assumption:results:regularity_of_V} holds.
	\end{proposition}
	
	\begin{remark}
		The above proposition is slightly more general than what we consider in Proposition \ref{prop:diffusion-coupled-to-jumps}, as there we assume that $F = \{1,\dots,J\}$ is a finite set.
	\end{remark}
	
	\begin{proof}[Proof]
		\underline{\ref{item:assumption:slow_regularity:continuity}}: Let $(x,p)\in E\times\mathbb{R}^d$. 
		Continuity of $\Lambda$ is a consequence of the fact that
		\begin{equation*}
			\Lambda(x,p,\theta) = \int_F V(x,p,z)\,d\theta(z)
		\end{equation*}
		is the pairing of a continuous and bounded function $V(x,p,\cdot)$ with the measure $\theta\in\mathcal{P}(F)$.
		
		\smallskip
		
		\underline{\ref{item:assumption:slow_regularity:convexity}}: Let $x\in E$ and $\theta \in \mathcal{P}(F)$. Convexity of $p\mapsto \Lambda(x,p,\theta)$ follows since $a(x,z)$ is positive definite by assumption. If $p_0 = 0$, then evidently $\Lambda(x,p_0,\theta) = 0$.
		\smallskip
		
		\underline{\ref{item:assumption:slow_regularity:compact_containment}}: We show that the map $\Upsilon : E\to\mathbb{R}$ defined by
		\begin{equation*}
			\Upsilon(x) := \frac{1}{2}\log\left(1 + |x|^2\right)
		\end{equation*}
		is a containment function for $\Lambda$. For any $x\in E$ and $\theta\in\mathcal{P}(F)$, we have
		\begin{align*}
			\Lambda(x,\nabla\Upsilon(x),\theta) &= \int_F \ip{a(x,z)\nabla\Upsilon(x)}{\nabla\Upsilon(x)}\,d\theta(z) + \int_F\ip{b(x,z)}{\nabla\Upsilon(x)}\,d\theta(z)\\
			&\leq a_{\{x\},\text{max}} |\nabla\Upsilon(x)|^2 + b_{\{x\},\text{max}}|\nabla\Upsilon(x)|\\
			&\leq C (1+|x|) \frac{x^2}{(1+x^2)^2} + C(1+|x|) \frac{x}{(1+x^2)},
		\end{align*}
		and the boundedness condition follows with the constant 
		\begin{equation*}
			C_\Upsilon := C \,\sup_x (1+|x|) \left[\frac{x^2}{(1+x^2)^2} + \frac{x}{(1+x^2)} \right] <\infty.
		\end{equation*}
		\smallskip
		
		\underline{\ref{item:assumption:slow_regularity:controlled_growth}}: Let $K\subseteq E$ be compact. We have to show that there exist constants $M, C_1, C_2 \geq 0$  such that for all $x \in K$, $p \in \mathbb{R}^d$ and all $\theta_1,\theta_2 \in \mathcal{P}(F)$, we have
		\begin{equation} \label{eqn:growth_bound_general}
			\Lambda(x,p,\theta_1)
			\leq
			\max \left\{M, C_1 \Lambda(x,p,\theta_2) + C_2 \right\}.
		\end{equation}
		Fix $\theta_1,\theta_2 \in \mathcal{P}(F)$.	We have for $x \in K$
		\begin{equation*}
			\int \ip{a(x,z)p}{p} d\theta_1(z) \leq \frac{a_{K,max}}{a_{K,min}} \int \ip{a(x,z)p}{p} d\theta_2(z)
		\end{equation*}
		In addition, as $a_{K,min} > 0$ and $b_{K,max} < \infty$ we have for any $C > 0$ and sufficiently large $|p|$ that
		\begin{equation*}
			\int \ip{b(x,z)}{p} \,d\theta_1(z) - (C+1)\int \ip{b(x,z)}{p} \,d\theta_2(z) \leq C \int \ip{a(x,z)p}{p} \,d\theta_2(z)
		\end{equation*}
		Thus, for sufficiently large $|p|$ (depending on $C$) we have
		\begin{equation*}
			\Lambda(x,p,\theta_1) \leq (1+C) \Lambda(x,p,\theta_2).
		\end{equation*}
		Fix a $C =: C_1$ and denote the set of `large' $p$ by $S$. The map $(x,p,\theta) \mapsto \Lambda(x,p,\theta)$ is bounded on $K \times \times S^c\times \Theta$. Thus, we can find a constant $C_2$ such that \eqref{eqn:growth_bound_general} holds.
		
		\smallskip
		
		\underline{\ref{item:assumption:slow_regularity:continuity_estimate}}: By the assumption on $a(x,z)$, the function $\Lambda$ is uniformly coercive in the sense that for any compact set $K\subseteq E$, 
		\begin{equation*}
			\inf_{x\in K, \theta\in\Theta}\Lambda(x,p,\theta) \to \infty \quad \text{ as }\; |p|\to \infty,
		\end{equation*}
		and the continuity estimate follows by Proposition~\ref{proposition:continuity_estimate_coercivity}.
	\end{proof}
	We proceed with the example in which $\Lambda$ depends on $p$ through exponential functions (Proposition~\ref{prop:mean-field-coupled-to-diffusion}). Let $q \in \mathbb{N}$ be an integer and 
	\begin{equation*}
		\Gamma := \left\{(a,b)\, \middle| \,a,b\in\{1,\dots,q\}, \,a\neq b\right\}
	\end{equation*}
	be the set of oriented edges in $\{1,\dots,q\}$.
	\begin{proposition}[Exponential function $\Lambda$]\label{prop:verify-ex:Lambda_exponential}
		Let $E\subseteq \mathbb{R}^d$ be the embedding of $E = \cP(\{1,\dots,q\}) \times (\bR^+)^{|\Gamma|}$ and $\Theta$ be a topological space. Suppose that $\Lambda$ is given by
		\begin{equation*} 
			\Lambda((\mu,w),p,\theta) = \sum_{(a,b) \in \Gamma} v(a,b,\mu,\theta)\left[\exp\left\{p_b - p _a + p_{(a,b)} \right\} - 1 \right]
		\end{equation*}
		where $v$ is a proper kernel in the sense of Definition~\ref{definition:proper_kernel}. Suppose in addition that there is a constant $C > 0$ such that for all $(a,b) \in \Gamma$ such that $v(a,b, \cdot,\cdot) \neq 0$ we have
		\begin{equation}\label{eq:prop-verify-Lambda-exp:boundedness-kernel}
			\sup_{\mu} \sup_{\theta_1,\theta_2} \frac{v(a,b,\mu,\theta_1)}{v(a,b,\mu,\theta_2)} \leq C.
		\end{equation}
		Then $\Lambda$ satisfies Assumption~\ref{assumption:results:regularity_of_V}.
	\end{proposition}
	\begin{remark}
		Similar to the previous proposition, the assumptions on $\Lambda$ are satisfied when $\Theta = \mathcal{P}(F)$ for some Polish space $F$, we have $v(a,b,\mu,\theta) = \mu(a) \int r(a,b,\mu,z) \theta(\dd z)$, and there are constants $0 < r_{min} \leq r_{max} < \infty$ such that for all $(a,b) \in \Gamma$ such that $\sup_{\mu,z} r(a,b,\mu,z) > 0$, we have
		\begin{equation*}
			r_{min} \leq \inf_{z} \inf_{\mu} r(a,b,\mu,z) \leq \sup_{z} \sup_{\mu} r(a,b,\mu,z) \leq r_{max}.
		\end{equation*}
		Regarding~\eqref{eq:prop-verify-Lambda-exp:boundedness-kernel}, for $(a,b) \in \Gamma$ for which $v(a,b,\cdot,\cdot)$ is non-trivial, we have
		\begin{equation*}
			\frac{v(a,b,\mu,\theta_1)}{v(a,b,\mu,\theta_2)} = \frac{\int r(a,b,\mu,z) \theta_1(\dd z)}{\int r(a,b,\mu,z) \theta_2(\dd z)} \leq \frac{r_{max}}{r_{min}}.
		\end{equation*}
	\end{remark}
	\begin{proof}[Proof of Proposition~\ref{prop:verify-ex:Lambda_exponential}]
		\underline{\ref{item:assumption:slow_regularity:continuity}}: The function $\Lambda$ is continuous as the sum of continuous functions.
		
		\smallskip
		
		\underline{\ref{item:assumption:slow_regularity:convexity}}: Convexity of $\Lambda$ as a function of $p$ follows from the fact that $\Lambda$ is a finite sum of convex functions, and $\Lambda(x,0,\theta)=0$ is evident.
		\smallskip
		
		\underline{\ref{item:assumption:slow_regularity:compact_containment}}: The function $\Upsilon : E\to\mathbb{R}$ defined by
		\begin{equation*}
			\Upsilon(\mu,w) := \sum_{(a,b)\in\Gamma}\log\left[1 + w_{(a,b)}\right]
		\end{equation*}
		is a containment function for $\Lambda$. For a verification, see~\cite{Kr17}.
		
		\smallskip
		
		\underline{\ref{item:assumption:slow_regularity:controlled_growth}}: Note that	
		\begin{align*}
			\Lambda((\mu,w),\theta_1,p) & \leq \sum_{(a,b)\in \Gamma} v(a,b,\mu,\theta_1) e^{p_{a,b} + p_b - p_a} \\
			& \leq C \sum_{(a,b)\in \Gamma} v(a,b,\mu,\theta_2) e^{p_{a,b} + p_b - p_a}  \\
			& \leq C \sum_{(a,b)\in \Gamma} v(a,b,\mu,\theta_2) \left[e^{p_{a,b} + p_b - p_a} - 1 \right] + C_2 .
		\end{align*}
		Thus the estimate holds with $M = 0$, $C_1 = C$ and $C_2 = \sup_{\mu,\theta} \sum_{a,b} v(a,b,\mu,\theta)$.
		
		\smallskip
		
		\underline{\ref{item:assumption:slow_regularity:continuity_estimate}}: The continuity estimate is the content of Proposition~\ref{proposition:continuity_estimate_directional_extended} below.
	\end{proof}

	
	\subsection{Verifying the continuity estimate} \label{section:verification_of_continuity_estimate}
	
	With the exception of  the verification of the continuity estimate in Assumption \ref{assumption:results:regularity_of_V} the verification in Section \ref{section:verify-ex:functions-Lambda} is straightforward. On the other hand, the continuity estimate is an extension of the comparison principle, and is therefore more complex. We verify the continuity estimate in three contexts, which illustrates that the continuity estimate follows from essentially the same arguments as the standard comparison principle. We will do this for:
	\begin{itemize}
		\item Coercive Hamiltonians
		\item One-sided Lipschitz Hamiltonians
		\item Hamiltonians arising from large deviations of empirical measures.
	\end{itemize}
	This list is not meant to be an exhaustive list, but to illustrate that the continuity estimate is a sensible extension of the comparison principle, which is satisfied in a wide range of contexts. In what follows, $E\subseteq \mathbb{R}^d$ is a Polish subset and $\Theta$ a topological space.
	\begin{proposition}[Coercive $\Lambda$] \label{proposition:continuity_estimate_coercivity}
		Let $\Lambda : E \times \bR^d \times \Theta \rightarrow \bR$ be continuous and uniformly coercive: that is, for any compact $K \subseteq E$ we have
		\begin{equation*}
			\inf_{x \in K, \theta\in\Theta} \Lambda(x,p,\theta) \to \infty \quad \mathrm{as} \; |p| \to \infty.
		\end{equation*}
		Then the continuity estimate holds for $\Lambda$ with respect to any penalization function $\Psi$.
	\end{proposition}
	
	\begin{proof}
		Let $\Psi(x,y) = \tfrac{1}{2}(x-y)^2$. Let $(x_{\alpha,\varepsilon},y_{\alpha,\varepsilon},\theta_{\varepsilon,\alpha})$ be fundamental for $\Lambda$ with respect to $\Psi$. Set $p_{\alpha,\varepsilon} = \alpha(x_{\varepsilon,\alpha} - y_{\varepsilon,\alpha})$. By the upper bound~\eqref{eqn:control_on_Gbasic_sup}, we find that for sufficiently small $\varepsilon > 0$ there is some $\alpha(\varepsilon)$ such that
		\begin{equation*}
			\sup_{\alpha \geq \alpha(\varepsilon)} \Lambda\left(y_{\varepsilon,\alpha}, p_{\varepsilon,\alpha}, \theta_{\varepsilon,\alpha}\right) < \infty.
		\end{equation*}
		As the variables $y_{\alpha,\varepsilon}$ are contained in a compact set by property (C1) of fundamental collections of variables, the uniform coercivity implies that the momenta $p_{\varepsilon,\alpha}$ for $\alpha \geq \alpha(\varepsilon)$ remain in a bounded set. Thus, we can extract a subsequence $\alpha'$ such that $(x_{\varepsilon,\alpha'},y_{\varepsilon,\alpha'},p_{\varepsilon,\alpha'},\theta_{\varepsilon,\alpha'})$ converges to $(x,y,p,\theta)$ with $x = y$ due to property (C2) of fundamental collections of variables. By continuity of $\Lambda$ we find
		\begin{align*}
			& \liminf_{\alpha \rightarrow \infty} \Lambda\left(x_{\varepsilon,\alpha}, p_{\varepsilon,\alpha},\theta_{\varepsilon,\alpha}\right) - \Lambda\left(y_{\alpha,\varepsilon},p_{\varepsilon,\alpha},\theta_{\varepsilon,\alpha}\right) \\
			& \leq \lim_{\alpha'\rightarrow \infty} \Lambda\left(x_{\varepsilon,\alpha'}, p_{\varepsilon,\alpha'},\theta_{\varepsilon,\alpha'}\right) - \Lambda\left(y_{\varepsilon,\alpha'},p_{\varepsilon,\alpha'},\theta_{\varepsilon,\alpha'}\right) = 0
		\end{align*}
		establishing the continuity estimate.
	\end{proof}

	\begin{proposition}[One-sided Lipschitz $\Lambda$]  \label{proposition:continuity_estimate_lipschitz}
		Let $\Lambda : E \times \bR^d \times \Theta\rightarrow \bR$ satisfy
		\begin{equation} \label{eqn:one_sided_Lipschitz_G}
			\Lambda(x,\alpha(x-y),\theta) - \Lambda(y,\alpha(x-y),\theta) \leq  c(\theta) \omega(|x-y| + \alpha (x-y)^2)
		\end{equation}
		for some collection of constants $c(\theta)$ satisfying $\sup_\theta c(\theta) < \infty$ and a function $\omega : \bR^+ \rightarrow \bR^+$ satisfying $\lim_{\delta \downarrow 0} \omega(\delta) = 0$.
		
		Then the continuity estimate holds for $\Lambda$ with respect to $\Psi(x,y) = \tfrac{1}{2}(x-y)^2$.
	\end{proposition}
	
	\begin{proof} [Proof]
		Let $\Psi(x,y) = \tfrac{1}{2}(x-y)^2$. Let $(x_{\alpha,\varepsilon},y_{\alpha,\varepsilon},\theta_{\varepsilon,\alpha})$ be fundamental for $\Lambda$ with respect to $\Psi$. Set $p_{\alpha,\varepsilon} = \alpha(x_{\varepsilon,\alpha} - y_{\varepsilon,\alpha})$. We find
		\begin{align*}
			& \liminf_{\alpha \rightarrow \infty} \Lambda\left(x_{\varepsilon,\alpha}, p_{\varepsilon,\alpha}, \theta_{\varepsilon,\alpha}\right) - \Lambda\left(y_{\alpha,\varepsilon},p_{\varepsilon,\alpha},\theta_{\varepsilon,\alpha}\right) \\
			& \leq \liminf_{\alpha\rightarrow \infty} c(\theta_{\varepsilon,\alpha}) \omega\left(|x_{\varepsilon,\alpha}-y_{\varepsilon,\alpha}| + \alpha (x_{\varepsilon,\alpha}-y_{\varepsilon,\alpha})^2\right)
		\end{align*}
		which equals $0$ as $\sup_\theta c(\theta) < \infty$, $\lim_{\delta \downarrow 0} \omega(\delta) = 0$ and property \ref{item:def:continuity_estimate:2}  of a fundamental collection of variables.
	\end{proof}

	For the empirical measure of a collection of independent processes one obtains maps $\Lambda$ that are neither uniformly coercive nor Lipschitz. Also in this context one can establish the continuity estimate. We treat a simple 1d case and then state a more general version for which we refer to \cite{Kr17}.

	\begin{proposition} \label{proposition:continuity_estimate_directional}
		Suppose that $E = [-1,1]$ and that $\Lambda(x,p,\theta)$ is given by
		\begin{equation*}
			\Lambda(x,p,\theta) = (1-x) c_+(\theta) \left[e^{p} -1\right] +  (1+x) c_-(\theta) \left[e^{-p} -1\right]
		\end{equation*}
		with $c_-,c_+$ non-negative functions of $\theta$.	Then the continuity estimate holds for $\Lambda$ with respect to $\Psi(x,y) = \tfrac{1}{2}(x-y)^2$.
	\end{proposition}

	\begin{proof}[Proof]
		Let $\Psi(x,y) = \tfrac{1}{2}(x-y)^2$. Let $(x_{\alpha,\varepsilon},y_{\alpha,\varepsilon},\theta_{\varepsilon,\alpha})$ be fundamental for $\Lambda$ with respect to $\Psi$. Set $p_{\alpha,\varepsilon} = \alpha(x_{\varepsilon,\alpha} - y_{\varepsilon,\alpha})$.
		
		We have
		\begin{align*}
			& \Lambda\left(x_{\varepsilon,\alpha}, p_{\varepsilon,\alpha}, \theta_{\varepsilon,\alpha}\right) - \Lambda\left(y_{\alpha,\varepsilon},p_{\varepsilon,\alpha},\theta_{\varepsilon,\alpha}\right) \\
			& = \left(y_{\varepsilon,\alpha}-x_{\varepsilon,\alpha}\right) c_+(\theta_{\varepsilon,\alpha}) \left[e^{p_{\varepsilon,\alpha}} -1\right] + \left(x_{\varepsilon,\alpha}-y_{\varepsilon,\alpha}\right) c_-(\theta_{\varepsilon,\alpha}) \left[e^{-p_{\varepsilon,\alpha}} -1\right]
		\end{align*}
		Now note that $y_{\varepsilon,\alpha}-x_{\varepsilon,\alpha}$ is positive if and only if $e^{p_{\varepsilon,\alpha}} -1$ is negative so that the first term is bounded above by $0$. With a similar argument the second term is bounded above by $0$. Thus the continuity estimate is satisfied.
	\end{proof}

	\begin{proposition} \label{proposition:continuity_estimate_directional_extended}
		Suppose $E = \cP(\{1,\dots,q\} \times (\bR^+)^\Gamma$ and suppose that $\Lambda$ is given by
		\begin{equation*} 
			\Lambda((\mu,w),\theta,p) = \sum_{(a,b) \in \Gamma} v(a,b,\mu,\theta)\left[\exp\left\{p_b - p _a + p_{(a,b)} \right\} - 1 \right]
		\end{equation*}
		where $v$ is a proper kernel. Then the continuity estimate holds for $\Lambda$ with respect to penalization functions (see Section \ref{section:continuity_estimate_general})
		\begin{align*}
			\Psi_1(\mu,\hat{\mu}) & := \frac{1}{2} \sum_{a} ((\hat{\mu}(a) - \mu(a))^+)^2, \\
			\Psi_2(w,\hat{w}) & := \frac{1}{2} \sum_{(a,b) \in \Gamma} (w_{(a,b)} - \hat{w}_{(a,b)})^2.
		\end{align*}
		Here we denote $r^+ = r \vee 0$ for $r \in \bR$.
	\end{proposition}
	
	In this context, one can use coercivity like in Proposition \ref{proposition:continuity_estimate_coercivity} in combination with directional properties used in the proof of Proposition \ref{proposition:continuity_estimate_directional} above.
	
	To be more specific: the proof of this proposition can be carried out exactly as the proof of Theorem 3.8 of \cite{Kr17}: namely at any point a converging subsequence is constructed, the variables $\alpha$ need to be chosen such that we also get convergence of the measures $\theta_{\varepsilon,\alpha}$ in $\cP(F)$.
	
	\subsection{Verifying Assumption \ref{assumption:Hamiltonian_vector_field} for the exponential internal Hamiltonian}
	

	\begin{proposition}\label{proposition:Ham_flow_for_exp_hamiltonian}
		Let $\Lambda$ be as in Proposition \ref{prop:mean-field-coupled-to-diffusion}:
		\begin{equation*} 
			\Lambda((\mu,w),p,\theta) = \sum_{(a,b) \in \Gamma} v(a,b,\mu,\theta)\left[\exp\left\{p_b - p _a + p_{(a,b)} \right\} - 1 \right]
		\end{equation*}
		Then we have $\partial_p \Lambda((\mu,x),p) \subseteq T_E(\mu,w)$.
	\end{proposition}

	\begin{proof}[A sketch of the verification of Assumption~\ref{assumption:Hamiltonian_vector_field}]
		We sketch the proof in a simplified case, the general case being similar. Consider~$E=\mathcal{P}(\{a,b\})$ (ignoring the flux for the moment), and identify~$E$ with the simplex in~$\mathbb{R}^2$. Fix the control $\theta \in \Theta$. We have to show~$\partial_p \Lambda(\mu,p,\theta) \subseteq T_E(\mu)$. Recall that~$T_E(\mu)$ is the tangent cone at~$\mu$, that means the vectors at~$\mu$ pointing inside of~$E$. We compute the vector~$\nabla_p \Lambda(\mu,p,\theta) \in\mathbb{R}^2$,
		\begin{equation*}
			\nabla_p \Lambda(\mu,p,\theta) = \begin{pmatrix}
				-v(a,b,\mu,\theta) e^{p_b-p_a} + v(b,a,\mu,\theta) e^{p_a-p_b}\\
				v(a,b,\mu,\theta) e^{p_b-p_a} - v(b,a,\mu,\theta) e^{p_a-p_b}
			\end{pmatrix}.
		\end{equation*}
		For~$\mu=(\mu_a,\mu_b)\in E$ with~$\mu_a,\mu_b > 0$, the tangent cone~$T_E(\mu)$ is spanned by~$(1,-1)^T$. Since~$\nabla_p \Lambda(\mu,p,\theta)$ is orthogonal to~$(1,1)^T$, we indeed find that~$\partial_p \Lambda(\mu,p,\theta) \subseteq T_E(\mu)$ in that case. For~$\mu=(1,0)$, the tangent cone is~$T_E(1,0)=\{\lambda(-1,1)^T\,:\,\lambda \geq 0\}$. We have
		\begin{equation*}
			\nabla_p \Lambda(\mu,p,\theta) = \begin{pmatrix}
				- v(a,b,\mu,\theta) e^{p_b-p_a}\\
				v(a,b,\mu,\theta) e^{p_b-p_a}
			\end{pmatrix},
		\end{equation*}
		which is parallel to~$(-1,1)^T$, and therefore~$\partial_p \Lambda(\mu,p,\theta) \subseteq T_E(\mu)$. The argument is similar for~$\mu=(0,1)$. The general case (including the fluxes) follows from a more tedious, but straightforward, computation.
	\end{proof}

	\subsection*{Acknowledgment}
	MS acknowledges financial support through NWO grant 613.001.552.
	\appendix

	\section{Viscosity solutions} \label{section:appendix_viscosity_solutions}
	In Section \ref{section:comparison_principle} we work with a pair of Hamilton-Jacobi equations instead of a single Hamilton-Jacobi equation. To this end, we need to extend the notion of a viscosity solution and that of the comparison principle of Section \ref{section:preliminaries}.

	\begin{definition} \label{definition:appendix_pair_ofHJ}
		Let $A_1 \subseteq C(E) \times C(E)$ and $A_2 \subseteq C(E) \times C(E)$. Fix $\lambda > 0$ and $h_1,h_2 \in C_b(E)$. Consider the equations
		\begin{align} 
			f - \lambda A_1 f & = h_1, \label{eqn:differential_equation_A1} \\
			f - \lambda A_2 f & = h_2. \label{eqn:differential_equation_A2}
		\end{align}

		We say that $u$ is a \textit{(viscosity) subsolution} of equation \eqref{eqn:differential_equation_A1} if $u$ is bounded, upper semi-continuous and if for all $(f,g) \in A_1$ there exists a sequence $x_n \in E$ such that
		\begin{gather*}
			\lim_{n \uparrow \infty} u(x_n) - f(x_n)  = \sup_x u(x) - f(x), \\
			\lim_{n \uparrow \infty} u(x_n) - \lambda g(x_n) - h(x_n) \leq 0.
		\end{gather*}
		We say that $v$ is a \textit{(viscosity) supersolution} of equation \eqref{eqn:differential_equation_A2} if $v$ is bounded, lower semi-continuous and if for all $(f,g) \in A_2$ there exists a sequence $x_n \in E$ such that
		\begin{gather*}
			\lim_{n \uparrow \infty} v(x_n) - f(x_n)  = \inf_x v(x) - f(x), \\
			\lim_{n \uparrow \infty} v(x_n) - \lambda g(x_n) - h(x_n) \geq 0.
		\end{gather*}
		If $h_1 = h_2$, we say that $u$ is a \textit{(viscosity) solution} of equations \eqref{eqn:differential_equation_A1} and \eqref{eqn:differential_equation_A2} if it is both a subsolution to \eqref{eqn:differential_equation_A1} and a supersolution to \eqref{eqn:differential_equation_A2}.
		
		We say that \eqref{eqn:differential_equation_A1} and \eqref{eqn:differential_equation_A2} satisfy the \textit{comparison principle} if for every subsolution $u$ to \eqref{eqn:differential_equation_A1} and supersolution $v$ to \eqref{eqn:differential_equation_A2}, we have $\sup_E u-v \leq \sup_E h_1 - h_2$.
	\end{definition}
	As before, if test functions have compact levelsets, the existence of a sequences can be replaced by the existence of a point.

	\section{Regularity of the Hamiltonian}
	\label{section:regularity-of-H-and-L}

	In this section, we establish continuity, convexity and the existence of a containment function for the Hamiltonian $\cH$ of \eqref{eq:results:variational_hamiltonian}. We repeat its definition for convenience:
	\begin{equation} \label{eqn:regularity_section_H_variational_rep}
		\mathcal{H}(x,p) = \sup_{\theta \in \Theta}\left[\Lambda(x,p,\theta) - \mathcal{I}(x,\theta)\right].
	\end{equation}
	\begin{proposition}[Regularity of the Hamiltonian]\label{prop:reg-of-H-and-L:reg-H}
		Let $\mathcal{H} : E \times \mathbb{R}^d\to \mathbb{R}$ be the Hamiltonian as in \eqref{eqn:regularity_section_H_variational_rep}, and suppose that Assumptions~\ref{assumption:results:regularity_of_V} and~\ref{assumption:results:regularity_I} are satisfied. Then:
		\begin{enumerate}[label=(\roman*)]
			\item \label{item:prop:reg-of-H-and-L:convex} For any $x \in E$, the map $p \mapsto \mathcal{H}(x,p)$ is convex and $\mathcal{H}(x,0) = 0$.
			\item \label{item:prop:reg-of-H-and-L:compact-contain} With the containment function $\Upsilon : E \to \mathbb{R}$ of~\ref{item:assumption:slow_regularity:compact_containment}, we have
			\[
			\sup_{x \in E}\mathcal{H}(x,\nabla\Upsilon(x)) \leq C_\Upsilon < \infty.
			\]
		\end{enumerate}
	\end{proposition}
	\begin{proof}[Proof]
		The map $p \mapsto \mathcal{H}(x,p)$ is convex as it is the supremum over convex (in $p$) functions.
		
		\smallskip
		
		For proving $\mathcal{H}(x,0) = 0$, let $x \in E$. Then by~\ref{item:assumption:slow_regularity:convexity} of Assumption~\ref{assumption:results:regularity_of_V}, we have $\Lambda(x,0,\theta) = 0$, and therefore
		\[
		\mathcal{H}(x,0) = - \inf_{\theta\in\Theta} \mathcal{I}(x,\theta) = 0,
		\]
		since $\cI \geq 0$ by Assumption \ref{assumption:results:regularity_I} and $\mathcal{I}(x,\theta_x^0)=0$ for some~$\theta_x^0$ by~\ref{item:assumption:I:zero-measure} of  Assumption~\ref{assumption:results:regularity_I}. Regarding~\ref{item:prop:reg-of-H-and-L:compact-contain}, we note that by~\ref{item:assumption:slow_regularity:compact_containment},
		\begin{align*}
			\mathcal{H}(x,\nabla \Upsilon(x)) \leq \sup_\theta \Lambda(x,\nabla \Upsilon(x),\theta) \leq \sup_{\theta\in\Theta}\sup_{x \in E} \Lambda(x,\nabla \Upsilon(x),\theta) \leq C_\Upsilon.
		\end{align*}
	\end{proof}
	To prove that $\cH$ is continuous, we use Assumption \ref{assumption:results:regularity_I}. What we truly need, however, is that $\cI$ Gamma converges as a function of $x$. We establish this result first.

	\begin{proposition}[Gamma convergence of the cost functions]\label{prop:Gamma-convergence-of-I}
		Let a cost function $\mathcal{I}:E\times\Theta\to[0,\infty]$ satisfy Assumption~\ref{assumption:results:regularity_I}. Then if $x_n\to x$ in $E$, the functionals $\mathcal{I}_n$ defined by
		\begin{equation*}
			\mathcal{I}_n(\theta) := \mathcal{I}(x_n,\theta)
		\end{equation*}
		converge in the $\Gamma$-sense to $\mathcal{I}_\infty(\theta) := \mathcal{I}(x,\theta)$. That is:
		\begin{enumerate}
			\item If $x_n \rightarrow x$ and $\theta_n \rightarrow \theta$, then $\liminf_{n\to\infty} \cI(x_n,\theta_n) \geq \cI(x,\theta)$,
			\item For $x_n \rightarrow x$ and all $\theta \in \Theta$ there are controls $\theta_n \in \Theta$ such that $\theta_n \rightarrow \theta$ and $\limsup_{n\to\infty} \cI(x_n,\theta_n) \leq \cI(x,\theta)$.
		\end{enumerate}
	\end{proposition}
	\begin{proof}[Proof]
		Let $x_n\to x$. If $\theta_n\to \theta$, then by lower semicontinuity~\ref{item:assumption:I:lsc},
		\begin{equation*}
			\liminf_{n\to\infty}\mathcal{I}(x_n,\theta_n) \geq \mathcal{I}(x,\theta).
		\end{equation*}
		For the $\text{lim-sup}$ bound, let $\theta\in\Theta$. If $\mathcal{I}(x,\theta)=\infty$, there is nothing to prove. Thus suppose that $\mathcal{I}(x,\theta)$ is finite. Then by~\ref{item:assumption:I:finiteness}, there is a neighborhood $U_x$ of $x$ and a constant $M < \infty$ such that for any $y\in U_x$,
		\begin{equation*}
			\mathcal{I}(y,\theta) \leq M.
		\end{equation*} 
		Since $x_n\to x$, the $x_n$ are eventually contained in $U_x$. Taking the constant sequence $\theta_n:=\theta$, we thus get that $\mathcal{I}(x_n,\theta_n) \leq M$ for all $n$ large enough. By~\ref{item:assumption:I:equi-cont}, 
		\begin{equation*}
			\lim_{n\to\infty}|\mathcal{I}(x_n,\theta_n)-\mathcal{I}(x,\theta)| \leq 0,
		\end{equation*}
		and the $\text{lim-sup}$ bound follows.
	\end{proof}
	\begin{proposition}[Continuity of the Hamiltonian]\label{prop:reg-of-H-and-L:continuity}
		Let $\mathcal{H} : E \times \mathbb{R}^d\to \mathbb{R}$ be the Hamiltonian defined in~\eqref{eq:results:variational_hamiltonian}, and suppose that Assumptions~\ref{assumption:results:regularity_of_V} and~\ref{assumption:results:regularity_I} are satisfied. Then the map $(x,p) \mapsto \cH(x,p)$ is continuous and the Lagrangian $(x,v) \mapsto \cL(x,v) := \sup_{p} \ip{p}{v} - \mathcal{H}(x,p)$ is lower semi-continuous.
	\end{proposition}

	Before we start with the proof, we give a remark on the generality of its statement and on the assumption that $\Theta$ is Polish.
	\begin{remark}
		The proof of upper semi-continuity of $\cH$ works in general, using continuity properties of $\Lambda$, lower semi-continuity of $(x,\theta) \mapsto \cI(x,\theta)$ and the compact sublevel sets of $\cI(x,\cdot)$. To establish lower semi-continuity, we need that the functionals $\cI$ Gamma converge as a function of $x$. This was established in Proposition \ref{prop:Gamma-convergence-of-I}.
	\end{remark}
	\begin{remark} 
		In the lemma we use a sequential characterization of upper hemi-continuity which holds if $\Theta$ is Polish. This is inspired by the natural formulation of Gamma convergence in terms of sequences. An extension of our results to spaces $\Theta$ beyond the Polish context should be possible to Hausdorff $\Theta$ that are k-spaces in which all compact sets are metrizable.
	\end{remark}
	We will use the following technical result to establish upper semi-continuity of $\cH$.
	\begin{lemma}[Lemma 17.30 in \cite{AlBo06}] \label{lemma:upper_semi_continuity_abstract}
		Let $\cX$ and $\cY$ be two Polish spaces. Let $\phi : \cX \rightarrow \cK(\cY)$, where $\cK(\cY)$ is the space of non-empty compact subsets of $\cY$. Suppose that $\phi$ is upper hemi-continuous, that is if $x_n \rightarrow x$ and $y_n \rightarrow y$ and $y_n \in \phi(x_n)$, then $y \in \phi(x)$. 
		
		Let $f : \text{Graph} (\phi) \rightarrow \bR$ be upper semi-continuous. Then the map $m(x) = \sup_{y \in \phi(x)} f(x,y)$ is upper semi-continuous.
	\end{lemma} 
	\begin{proof}[Proof of Proposition~\ref{prop:reg-of-H-and-L:continuity}]
		We start by establishing upper semi-continuity of $\cH$. We argue on the basis of Lemma \ref{lemma:upper_semi_continuity_abstract}. Recall the representation of $\cH$ of \eqref{eqn:regularity_section_H_variational_rep}. Set $\mathcal{X} = E\times\mathbb{R}^d$ for the $(x,p)$ variables, $\mathcal{Y} = \Theta$, and $f(x,p,\theta) = \Lambda(x,p,\theta) - \cI(x,\theta)$ and note that this function is upper semi-continuous by Assumption~\ref{assumption:results:regularity_I}~\ref{item:assumption:I:lsc} and by Assumption \ref{assumption:results:regularity_of_V}~\ref{item:assumption:slow_regularity:continuity}.

		\smallskip
		
		By Assumption \ref{assumption:results:regularity_I}~\ref{item:assumption:I:zero-measure}, we have $\cH(x,p) \geq \Lambda(x,p,\theta_{x}^0)$, where $\theta_{x}^0$ is a control such that $\cI(x,\theta_{x}^0) = 0$.  Thus, it suffices to restrict the supremum over $\theta \in \Theta$ to $\theta \in \phi(x,p)$ where 
		\begin{equation*}
			\phi(x,p) := \left\{\theta \in \Theta \, \middle| \, \cI(x,\theta) \leq 2 \vn{\Lambda(x,p,\cdot)}_\Theta \right\},
		\end{equation*}
		where $\vn{\cdot}_\Theta$ denotes the supremum norm on $\Theta$. Note that $\vn{\Lambda(x,p,\cdot)}_{\Theta} < \infty$ by Assumption \ref{assumption:results:regularity_of_V} \ref{item:assumption:slow_regularity:controlled_growth}. It follows that
		\begin{equation*}
			\mathcal{H}(x,p) = \sup_{\theta\in \phi(x,p)}\left[\Lambda(x,p,\theta)-\mathcal{I}(x,\theta)\right].
		\end{equation*}
		$\phi(x,p)$ is non-empty as $\theta_{x}^0 \in \phi(x,p)$ and it is compact due to Assumption \ref{assumption:results:regularity_I}~\ref{item:assumption:I:compact-sublevelsets}. We are left to show that $\phi$ is upper hemi-continuous.
		
		\smallskip
		
		Thus, let $(x_n,p_n,\theta_n) \rightarrow (x,p,\theta)$ with $\theta_n \in \phi(x_n,p_n)$. We establish that $\theta \in \phi(x,p)$. By \ref{item:assumption:I:lsc} and the definition of $\phi$ we find
		\begin{equation*}
			\cI(x,\theta) \leq \liminf_n \cI(x_n,\theta_n) \leq \liminf_n 2\vn{\Lambda(x_n,p_n,\cdot}_\Theta = 2 \vn{\Lambda(x,p,\cdot)}_\Theta
		\end{equation*}
		which implies indeed that $\theta  \in \phi(x,p)$. Thus, upper semi-continuity follows by an application of Lemma \ref{lemma:upper_semi_continuity_abstract}.

		\smallskip
		
		We proceed with proving lower semi-continuity of $\cH$. Suppose that $(x_n,p_n) \rightarrow (x,p)$, we prove that $\liminf_n \cH(x_n,p_n) \geq \cH(x,p)$.

		Let $\theta$ be the measure such that $\cH(x,p) = \Lambda(x,p,\theta) - \cI(x,\theta)$. We have
		\begin{itemize}
			\item By Proposition \ref{prop:Gamma-convergence-of-I} there are $\theta_n$ such that $\theta_n \rightarrow \theta$ and $\limsup_n \cI(x_n,\theta_n) \leq \cI(x,\theta)$. 
			\item $\Lambda(x_n,p_n,\theta_n)$ converges to $\Lambda(x,p,\theta)$ by Assumption \ref{item:assumption:slow_regularity:continuity}.
		\end{itemize}
		Therefore,
		\begin{align*}
			\liminf_{n\to\infty}\mathcal{H}(x_n,p_n)&\geq \liminf_{n\to\infty} \left[\Lambda(x_n,p_n,\theta_n)-\mathcal{I}(x_n,\theta_n)\right]\\
			&\geq \liminf_{n\to\infty}\Lambda(x_n,p_n,\theta_n)-\limsup_{n\to\infty}\mathcal{I}(x_n,\theta_n)\\
			&\geq \Lambda(x,p,\theta)-\mathcal{I}(x,\theta) = \mathcal{H}(x,p),
		\end{align*}
		establishing that $\mathcal{H}$ is lower semi-continuous.
		
		\smallskip
		
		The Lagrangian $\cL$ is obtained as the supremum over continuous functions. This implies $\cL$ is lower semi-continuous.
	\end{proof}

	\section{A more general continuity estimate} \label{section:continuity_estimate_general}

	In classical literature, the comparison principle for the Hamilton-Jacobi equation $f - \lambda Hf = h$ is often proven using a squared distance as a penalization function. This often works well due to the quadratic structure of the Hamiltonian. In different contexts, e.g. for the Hamiltonians arising from the large deviations of jump processes, this is not natural, see the issues arising in the proofs in \cite{DuIiSo90,Kr17}. In absence of a general method to solve these issues, ad-hoc procedures can be introduced. One such ad-hoc procedure introduced in \cite{Kr17} is to work with multiple penalization functions (in that context $\{\Psi_1,\Psi_2\}$) that explore different parts of the state-space.
	
	\smallskip
	
	Any argument that has been carried out in the main text can be carried out with the generalization of the continuity estimate below.
	
	\begin{definition}
		We say that $\{\Psi_1,\Psi_2\}$, $\Psi_i : E^2 \rightarrow \bR^+$ is a \textit{pair of penalization functions} if $\Psi_i \in C^1(E^2)$ and if $x = y$ if and only if $\Psi_i(x,y) = 0$ for all $i$.
	\end{definition}

	\begin{definition}[Continuity estimate] \label{def:fundamental_inequality_extended}
		Let $\cG: E \times\mathbb{R}^d \times \Theta \rightarrow \bR$, $ (x,p,\theta)\mapsto \cG(x,p,\theta)$ be a function and $\{\Psi_1,\Psi_2\}$ be a pair of penalization functions. Suppose that for each $\varepsilon > 0$ there is a sequence $\alpha_2 \rightarrow \infty$. As before, we suppress the dependence on $\varepsilon$. Suppose that for each $\varepsilon$ and $\alpha_2$ , there is a sequence $\alpha_1 \rightarrow \infty$. We suppress writing the dependence of the sequence $\alpha_1$ on $\varepsilon$ and $\alpha_2$. We write $\alpha = (\alpha_1,\alpha_2)$.
		
		Suppose that for each triplet $(\varepsilon,\alpha_1,\alpha_2)$ as above we have variables $(x_{\varepsilon,\alpha},y_{\varepsilon,\alpha})$ in $E^2$ and variables $\theta_{\varepsilon,\alpha}$ in $\Theta$. We say that this collection is \textit{fundamental} for $\cG$ with respect to $\{\Psi_1,\Psi_2\}$ if:
		\begin{enumerate}[(C1)]
			\item  For each $\varepsilon$, there are compact sets $K_\varepsilon \subseteq E$ and $\widehat{K}_\varepsilon\subseteq\Theta$ such that for all $\alpha$ we have $x_{\varepsilon,\alpha},y_{\varepsilon,\alpha} \in K_\varepsilon$ and $\theta_{\varepsilon,\alpha}\in\widehat{K}_\varepsilon$.
			\item For each $\varepsilon > 0$ and $\alpha_2$ there are limit points $x_{\varepsilon,\alpha_2}, y_{\varepsilon,\alpha_2} \in K_\varepsilon$ of $x_{\varepsilon,\alpha}$ and $y_{\varepsilon,\alpha}$ as $\alpha_1 = \alpha_1(\varepsilon,\alpha_2) \rightarrow \infty$. For each $\varepsilon$ there are limit points $x_\varepsilon,y_\varepsilon$ in $K_\varepsilon$ of $x_{\varepsilon,\alpha_2}$ and $y_{\varepsilon,\alpha_2}$ as $\alpha_2 \rightarrow \infty$. We furthermore have
			\begin{align*}
				& \Psi_1(x_{\varepsilon,\alpha_2},y_{\varepsilon,\alpha_2}) = 0 && \forall \, \varepsilon >0, \, \forall \, \alpha_2, \\
				& \Psi_1(x_\varepsilon,y_\varepsilon) + \Psi_2(x_\varepsilon,y_\varepsilon) = 0, && \forall \, \varepsilon > 0, \\
				& \lim_{\alpha_1 \rightarrow \infty} \alpha_1 \Psi_1(x_{\varepsilon,\alpha_1,\alpha_2},x_{\varepsilon,\alpha_1,\alpha_2}) = 0, && \forall  \, \varepsilon >0, \, \forall \, \alpha_2, \\
				& \lim_{\alpha_2 \rightarrow \infty} \alpha_2 \Psi_2(x_{\varepsilon,\alpha_1},x_{\varepsilon,\alpha_1}) = 0, && \forall \, \varepsilon > 0.
			\end{align*}
			\item We have
			\begin{align} 
				& \sup_{\alpha_2} \sup_{\alpha_1} \cG\left(y_{\varepsilon,\alpha}, - \sum_{i=1}^2\alpha_i (\nabla \Psi_i(x_{\varepsilon,\alpha},\cdot))(y_{\varepsilon,\alpha}),\theta_{\varepsilon,\alpha}\right) < \infty, \label{eqn:control_on_Gbasic_sup_extended} \\
				& \inf_{\alpha_2} \inf_{\alpha_1} \cG\left(x_{\varepsilon,\alpha}, \sum_{i=1}^2\alpha_i (\nabla \Psi_i(\cdot,y_{\varepsilon,\alpha}))(y_{\varepsilon,\alpha}),\theta_{\varepsilon,\alpha}\right) > - \infty. \label{eqn:control_on_Gbasic_inf_extended} 	
			\end{align} \label{itemize:funamental_inequality_control_upper_bound_extended}
			In other words, the operator $\cG$ evaluated in the proper momenta is eventually bounded from above and from below.
		\end{enumerate}
		We say that $\cG$ satisfies the \textit{continuity estimate} if for every fundamental collection of variables we have for each $\varepsilon > 0$ that
		\begin{multline}\label{equation:Xi_negative_liminf_extended}
			\liminf_{\alpha_2 \rightarrow \infty} \liminf_{\alpha_1 \rightarrow \infty} \cG\left(x_{\varepsilon,\alpha},\sum_{i=1}^2 \alpha_i \nabla \Psi_i(\cdot,y_{\varepsilon,\alpha})(x_{\varepsilon,\alpha}),\theta_{\varepsilon,\alpha}\right) \\
			- \cG\left(y_{\varepsilon,\alpha},- \sum_{i=1}^2 \alpha_i \nabla \Psi_i(x_{\varepsilon,\alpha},\cdot)(y_{\varepsilon,\alpha}),\theta_{\varepsilon,\alpha}\right) \leq 0.
		\end{multline}
	\end{definition}

	\section{Differential inclusions} \label{appendix:differential_inclusions}
	
	To establish that Condition 8.11 of \cite{FK06} is satisfied in the proof of Theorem \ref{theorem:existence_of_viscosity_solution}, we need to solve a differential inclusion. The following appendix is based on \cite{De92,Ku00} and is a copy of the one in \cite{KrMa20}. We state it for completeness.
	
	\smallskip
	
	Let $D \subseteq \bR^d$ be a non-empty set. A \textit{multi-valued mapping} $F : D \rightarrow 2^{\bR^d} \setminus \{\emptyset\}$ is a map that assigns to every $x \in D$ a set $F(x) \subseteq \bR^d$, $F(x) \neq \emptyset$.
	
	\begin{definition}
		Let $I \subseteq \bR$ be an interval with $0 \in I$, $D\subseteq \bR^d$, $x \in D$ and $F : D \rightarrow 2^{\bR^d} \setminus \emptyset$ a multi-valued mapping. A function $\gamma$ such that
		\begin{enumerate}[(a)]
			\item $\gamma : I \rightarrow D$ is absolutely continuous,
			\item $\gamma(0) = x$,
			\item $\dot{\gamma}(t) \in F(\gamma(t))$ for almost every $t \in I$
		\end{enumerate}
		is  called a \textit{solution of the differential inclusion} $\dot{\gamma} \in F(\gamma)$ a.e., $\gamma(0) = x$.
	\end{definition}
	
	If we assume sufficient regularity on the multi-valued mapping $F$, we can ensure the existence of a solution to differential inclusions that remain inside $D$.
	
	\begin{definition}
		Let $D \subseteq \bR^d$ be a non-empty set and let $F : D \rightarrow 2^{\bR^d} \setminus \{\emptyset\}$ be a multi-valued mapping.
		\begin{enumerate}[(i)]
			\item We say that $F$ is \textit{closed, compact or convex valued} if each set $F(x)$, $x \in D$ is closed, compact or convex, respectively.
			\item We say that $F$ is \textit{upper hemi-continuous at $x \in D$} if for each neighbourhood $\cU$ of $F(x)$, there is a neighbourhood $\cV$ of $x$ in $D$ such that $F(\cV) \subseteq \cU$. 
			We say that $F$ is \textit{upper hemi-continuous} if it is upper hemi-continuous at every point. $F$ is upper hemi-continuous if and only if for each sequence $x_n \rightarrow x$ in $D$ and $\xi_n \in F(x_n)$ such that $\xi_n \rightarrow \xi$ we have $\xi \in F(x)$.
		\end{enumerate}
	\end{definition}
	
	\begin{definition}
		Let $D \subseteq \bR^d$ be a closed non-empty set. The tangent cone to $D$ at $x$ is
		\begin{equation*}
			T_D(x) := \left\{z \in \bR^d \, \middle| \, \liminf_{\lambda \downarrow 0} \frac{d(y + \lambda z, D)}{\lambda} = 0\right\}.
		\end{equation*}
		The set $T_D(x)$ is sometimes called the  the \textit{Bouligand cotingent cone}.
	\end{definition}
	
	\begin{lemma}[Proposition 4.1 in \cite{De92}] \label{proposition:tangent_cone_convex}
		Let $D \subseteq \bR^d$ be a closed, convex, non-empty set. Then the set $T_D(x)$ is convex and contains $0$.
	\end{lemma}

	\begin{lemma}[Theorem 2.2.1 in \cite{Ku00}, Lemma 5.1 in \cite{De92}] \label{lemma:solve_differential_inclusion}
		Let $D \subseteq \bR^d$ be closed and let $F : D \rightarrow 2^{\bR^d} \setminus \{\emptyset\}$ satisfy
		\begin{enumerate}[(a)]
			\item $F$ has closed convex values and is upper hemi-continuous;
			\item for every $x$, we have $F(x) \cap T_D(x) \neq \emptyset$;
			\item $F$ has bounded growth: there is some $c > 0$ such that $\vn{F(x)} = \sup\left\{|z| \, \middle| \, z \in F(x) \right\} \leq c(1 + |x|)$ for all $x \in D$.
		\end{enumerate}
		Then the differential inclusion $\dot{\gamma} \in F(\gamma)$ has a solution on $\bR^+$ for every starting point $x \in D$.
	\end{lemma}

	\section{Pseudo-coercivity for exponential Hamiltonians} \label{appendix:pseudo_coercive}
	
	In this section we consider the notions of pseudo-coercivity and the continuity estimate for the Hamiltonian-Jacobi equation $f - \Lambda f = h$ on $E = [0,\infty)$, with $h \in C_b(E)$ for the Hamiltonian $\Lambda f(x) = \Lambda (x,f'(x))$ with
	\begin{equation} \label{eqn:one_sided_jumps}
		\Lambda(x,p) = x \left[e^{-p} - 1\right]
	\end{equation} 
	which is a simplified version of the Hamiltonian given in the introduction and Proposition \ref{prop:mean-field-coupled-to-diffusion}. 
	
	\smallskip

	The pseudo-coercivity estimate of \cite[Pages 34 and 35]{Ba94} translates into the present context as
	\begin{equation} \label{eqn:pseudo_coercivity_estimate}
		\left|\Lambda(x,p) - \Lambda(y,p) \right| \leq m\left(|x-y|(1+|p|)\right) Q(x,y,p)
	\end{equation}
	where $Q(x,y,p) = \max\left(\Psi(H(x,p)),\Psi(H(y,p))\right)$, where $m : [0,\infty) \rightarrow [0,\infty)$ is such that $\lim_{t \downarrow 0} m(t) = 0$ and $\Psi : \bR \rightarrow [0,\infty)$ is continuous.
	
	\smallskip
	
	We first make some general remarks on the relation between pseudo-coercivity and the continuity estimate. We then show that for \ref{eqn:one_sided_jumps}, pseudo-coercivity fails, whereas the continuity estimate holds (in the case of a single $\theta$).
	
	\smallskip
	
	The continuity estimate (for a single $\theta$) is more general than pseudo-coercivity in the sense that:
	\begin{itemize}
		\item It does not rely on the fact that you use a multiple of $|x-y|^2$ as a penalization in the comparison principle. This is of importance for Hamiltonians $\Lambda$ on e.g. the set of probability measures $\cP(\{1,\dots,q\})$ with $q \in \{3,4,\dots\}$ like in Proposition \ref{proposition:continuity_estimate_directional_extended}.
		\item It removes the necessity of taking absolute values in the estimate. This last fact is important as can be seen for the Hamilton-Jacobi equation for $f - \lambda \Lambda f = h$, $h \in C_b$, $\lambda > 0$ for the Hamiltonian $\Lambda f(x) = \Lambda(x,f'(x))$ with 
		\begin{equation*}
			\Lambda(x,p) = x \left[e^{-p} - 1\right], \qquad x \in [0,\infty), p \in \bR,
		\end{equation*}
		for which the comparison principle holds, whereas for $\widetilde{\Lambda}(x,p) := \Lambda(x,-p)$ it fails. Our continuity estimate holds for $\Lambda$, but not for $\widetilde{\Lambda}$. Pseudo-coercivity fails for both as we explain below.
	\end{itemize}

	We first show that pseudo-coercivity fails for the Hamilton-Jacobi equation in terms of $\Lambda$ of \eqref{eqn:one_sided_jumps}.

	\begin{lemma}
		$\Lambda$ is not pseudo-coercive.
	\end{lemma}
	
	\begin{proof}
		A counterexample suffices. 
		
		Let $W$ be the Lambert function. That is, $W$ is the inverse of $\phi$ where $\phi(x) = xe^x$. Next, let $x_\alpha = 0$, $y_\alpha = \alpha^{-1} W(\alpha)$. Then we have $y_\alpha \rightarrow 0$, $\alpha y_\alpha \rightarrow \infty$, $\alpha y_\alpha^2 = \frac{W(\alpha)^2}{\alpha} \rightarrow 0$ and
		\begin{equation*}
			y_\alpha e^{ \alpha y_\alpha} = 1.
		\end{equation*}
		Thus, with~$p_\alpha=\alpha(x_\alpha-y_\alpha)=-\alpha y_\alpha$,
		\begin{equation} \label{eqn:estimate_one_sided_jumps}
			\Lambda(x_\alpha,p_\alpha) - \Lambda(y_\alpha,p_\alpha) = y_\alpha -1 \rightarrow - 1
		\end{equation}
		contradicting \eqref{eqn:pseudo_coercivity_estimate}.
	\end{proof}

	Note that the comparison principle for $f - \lambda \Lambda f = h$ does in fact hold. This is due to the fact that one only needs to establish that $\liminf_{\alpha \rightarrow \infty} \Lambda(x_\alpha,p_\alpha) - \Lambda(y_\alpha,p_\alpha) \leq 0$ for appropriately chosen $x_\alpha,y_\alpha,p_\alpha$ without absolute value signs, see Proposition  \ref{proposition:continuity_estimate_directional}. The removal of absolute value signs is essential: the comparison principle for $f - \widetilde{\Lambda} f = h$ for
	\begin{equation*}
		\widetilde{\Lambda} (x,p) = x \left[e^p - 1\right]
	\end{equation*}
	fails. This is related to the statement that an associated large deviation principle fails, see Example E of \cite{ShWe05}.

	
\bibliographystyle{abbrv}
	\bibliography{KraaijBib}

\end{document}